\documentclass[12pt,reqno]{amsart}
\usepackage{amssymb,amsmath,dsfont}
\usepackage{pgfplots}
\pgfplotsset{compat=1.14}
\usepackage{enumitem}
\usepackage{calc}

\title{The Asynchronous DeGroot Dynamics}

\author{Dor Elboim}
\author{Yuval Peres}
\author{Ron Peretz}
\address{Dor Elboim, Princeton University, Princeton NJ, United States}
\email{delboim@princeton.edu}
\address{Yuval Peres, Beijing Institute of Mathematical Sciences and Applications, Beijing, China}
\email{yuval@yuvalperes.com}
\address{Ron Peretz, Bar Ilan University, Israel}
\email{ron.peretz@biu.ac.il}

\usepackage[margin=1in]{geometry}
\usepackage[all]{xy}
\usepackage{amssymb}
\usepackage{amsfonts}
\usepackage{amsthm}
\usepackage{amsmath}
\usepackage{hyperref}
\usepackage{mathtools}
\usepackage{url}
\usepackage{xcolor}
\usepackage{dsfont}
\usepackage{pgfplots}

\newtheorem{thm}{Theorem}[section]

\newtheorem{lem}[thm]{Lemma}  
\newtheorem{prop}[thm]{Proposition}
\newtheorem{cor}[thm]{Corollary}

\newtheorem{claim}[thm]{Claim}
\newtheorem{conjecture}[thm]{Conjecture}

\theoremstyle{definition}
\newtheorem{assumption}{Assumption}
\newtheorem{remark}[thm]{Remark}
\newtheorem{definition}[thm]{Definition}

\newcommand{\RR}{\mathbb{R}}
\newcommand{\PP}{\mathbb{P}}
\newcommand{\wP}{\widetilde{P}}

\newcommand{\EE}{\mathbb{E}}

\newcommand{\cF}{\mathcal{F}}
\newcommand{\cE}{\mathcal{E}}

\newcommand{\eps}{\varepsilon}
\newcommand{\var}{\mathrm{Var}}
\newcommand{\vpi}{\mathrm{var}_\pi}

\newcommand*\diff{\mathop{}\!\mathrm{d}}

\mathtoolsset{showonlyrefs}

\numberwithin{equation}{section}

\begin{document}
\maketitle
\begin{abstract}
    We analyze the asynchronous version of the DeGroot dynamics: In a connected graph $G$ with $n$ nodes,  each node has an initial opinions in $[0,1]$ and an independent Poisson clock. When a clock at a node $v$ rings,  the opinion at $v$ is replaced by the average opinion of its neighbors. It is well known that the opinions converge to a consensus.  We show that  the expected time $\EE(\tau_\eps)$ to reach $\eps$-consensus is poly$(n)$ in undirected graphs and in Eulerian digraphs, but for some digraphs of bounded degree it is exponential.  
    
     Our main result is that in undirected graphs and Eulerian digraphs, if the degrees are uniformly bounded and the initial opinions are i.i.d., then $\EE(\tau_\eps)=\text{polylog}(n)$ for every fixed $\eps>0$. We give sharp estimates for the variance of the limiting consensus opinion, which measures the ability to aggregate information (``wisdom of the crowd''). We also prove generalizations to non-reversible Markov chains and infinite graphs.   New results of independent interest on fragmentation processes and coupled random walks are crucial to our analysis.  
\end{abstract}

\section{Introduction}
The DeGroot dynamics \cite{degroot1974reaching} is arguably the most prominent model of non-Bayesian social learning. According to the classic DeGroot dynamics, agents synchronously update their opinions at discrete time periods assuming a weighted average opinion of their neighbors. We analyze a natural variant of the classic model, the asynchronous DeGroot dynamics, according to which  agents update their opinions at the rings of independent Poisson clocks.

The classic DeGroot model has been studied extensively (see  survey \cite{golublearning}).
It seems realistic to assume that not all agents update their opinion at the very same time, but rather do so at different random times. However, few quantitative results for the asynchronous dynamics have been proved.  

A fundamental difference between the asynchronous DeGroot dynamics and many other averaging dynamics, is that in the other dynamics the (weighted) average of the opinions never changes, whereas, in the asynchronous DeGroot dynamics such an invariant does not exist. This feature makes the analysis of the DeGroot dynamics harder. It also raises questions about the distribution of the consensus which do not arise  in the other models.

Most of the results regarding the classic model stem from its immediate relation to a well studied mathematical object, Markov chains. The analysis of the asynchronous counterpart is more subtle. In addition to Markov chain methods, we employ the connection to a   fragmentation process  of independent interest.

\subsection{Model and results}
Let $V$ be a finite or countable set of vertices. A \emph{network} consists of a stochastic matrix $P$ with rows and columns indexed by $V$.
The asynchronous DeGroot process associated with $P$ is defined as follows. Each vertex $v\in V$ holds an opinion $f_t(v)\in\mathbb R$ at every time $t\in \mathbb R_+$. At time $0$, the  initial opinions $\{f_0(v)\}_{v\in V}$ are either deterministic, or random. Each vertex $v$ is endowed with an independent  Poisson clock of rate $1$. 
At a time $t$ in which the clock of vertex $v$ rings, its opinion is updated through the DeGroot updating rule
\begin{equation}\label{eq:degroot}
f_t(v):=\sum_{u\in V}P_{vu}f_{t-}(u),    
\end{equation}
where $f_{t-}(u):=\lim_{s\nearrow t}f_s(u)$. 

A special case of interest is when $P$ is the transition matrix of a simple random walk on a (locally) finite undirected graph $G=(V,E)$. 
 In this case, each time a vertex rings, it updates its opinion to be the average of the opinions of its neighbors. We refer to this dynamics as \emph{the DeGroot dynamics on $G$}.

The following known theorem\footnote{Theorem~\ref{theorem: consensus} is an immediate consequence of a more general result that is described in \cite[Proposition 9]{golublearning} and attributed to \cite{chatterjee1977towards}. It also follows from Theorem~\ref{genrate}, Part (i).} is the asynchronous version of 
DeGroot's classical result~\cite{degroot1974reaching}.

\begin{thm}\label{theorem: consensus}
If $V$ is finite and $P$ is irreducible,
 then there exists a random variable $f_\infty$ such that for every $v\in V$, we have $\lim_{t\to\infty}f_t(v)=f_\infty$ almost surely.
\end{thm}

The focus of this paper is two fold:  studying the rate of convergence and the concentration of $f_\infty$ in finite graphs and generalizing  the convergence result  to infinite graphs.

When $V$ is finite, we have $\max\limits_{v\in V}f_t(v)\searrow f_\infty$ and $\min\limits_{v\in V}f_t(v)\nearrow f_\infty$. The $\eps $-consensus time of the dynamics is defined as the stopping time 
\[
\tau_\eps:=\min \Big\{ t:\max_{v\in V}f_t(v)-\min_{v\in V}f_t(v)\leq \eps \Big\}.
\]
In general, $\mathbb E[\tau_\eps]$ can be exponential in $|V|$. We find classes of stochastic matrices for which $\mathbb E[\tau_\eps]$ is polynomial in $|V|$, e.g., simple random walks on graphs.
\begin{thm}\label{cor:yuval}
Consider the DeGroot dynamics where the matrix $P$ represents a simple random walk on a connected graph $G$ with $n$ vertices and maximal degree $\Delta$. Suppose that $\Vert f_0\Vert_{\infty}\leq 1$. Then,
\begin{itemize}
\item[{\rm (i)}]  $\mathbb E [\tau_\eps]\le 4\Delta n^2 \lceil\log_2(1/\eps)\rceil$,
  \item[{\rm (ii)}] $\displaystyle \var(f_\infty) \le \Delta/|E|$.
  \end{itemize}
 If, in addition, $G$ is regular, then the factor $\Delta$ in  {\rm (i)} can be replaced by $3$.
\end{thm}
In Section~\ref{sec:finite} we prove refinements of these bounds that take into account the diameter of $G$,  its eigenvalues, and the initial configuration. Specifically,  Part (i) of Theorem~\ref{cor:yuval} follows from Theorem \ref{rate}, part (ii) from Corollary \ref{thm:varcons1}, and the last sentence from Remark (1) after Theorem~\ref{rate}.

If the graph $G$ is directed, the time to consensus
can be exponential (See Claim~\ref{claim:exp}), but when $G$ is a directed Eulerian graph (indegree equals outdegree at each vertex) we prove a polynomial upper bound in Corollary~\ref{cor:euler}.

Our main contribution is showing that when the initial opinions are i.i.d., rather than arbitrary, and the graph has a bounded degree, convergence to consensus is much faster if $n$ is large enough.
We give a polylogarithmic upper bound for the consensus time in this setting. 
\begin{thm}\label{cor:polylog}
Let $G$ be either a connected undirected graph or an Eulerian directed graph. Suppose that $G$ has $n$ 
 vertices and maximal degree $\Delta$. Consider the DeGroot dynamics on $G$ with initial opinions being $[0,1]$-valued i.i.d.\ random variables. Then, there exists a universal constant $C>0$ such that for any $\eps >0$
\[
\mathbb E[\tau_\eps]\leq (\eps^{-1} \log n)^C,
\]
as long as  $n$ is sufficiently large depending on $\Delta$.
\end{thm}
Without the assumption of bounded degree, the consensus time can be polynomial in $n$ (see Claim~\ref{claim:n^2}). 
Theorem~\ref{cor:polylog} follows from a more general result, Theorem~\ref{thm:polylog}, and Remark~\ref{remark: polylog expectation}. The polylogarithmic upper bound given in Theorem~\ref{cor:polylog} cannot be improved in general. Indeed,  later in Claim~\ref{claim:lower bound}, we show that on the cycle graph, $\mathbb E[\tau _{\eps }]\ge C\eps ^{-4} \log ^2 n$ for any $\eps \ge n^{-1/9}$.
 
Theorem~\ref{theorem: consensus} does not hold for infinite networks as stated. Claim~\ref{claim:initial} shows that for certain initial opinions convergence may fail. The following theorem provides a class of infinite networks and initial opinions to which Theorem~\ref{theorem: consensus} does extend. A more general result is given in Theorem~\ref{thm: convergence infinite network}.

\begin{thm}\label{cor: infinite graph}
Consider the DeGroot dynamics on an infinite, connected graph of bounded degree and suppose that the initial opinions are bounded i.i.d.\ random variables with expectation $\mu $. Then, for every $v\in V$
\[
\lim_{t\to\infty}f_t(v)=\mu\quad \quad\text{almost surely.}
\]
\end{thm}

The assumption of a bounded degree cannot be removed. Claim~\ref{claim:unbounded} provides an example of a graph of unbounded degree on which the opinions do not converge, even when the initial opinions are bounded i.i.d.\ random variables. 

\begin{remark}
On finite networks, the DeGroot dynamics is well defined through the updating rule \eqref{eq:degroot}. On infinite networks, a more subtle definition is required which is given in Definition~\ref{def: degroot}.
\end{remark}

\subsection{Related work}
Several asynchronous averaging dynamics have been studied, e.g., dynamics in which an edge is chosen randomly and then the two vertices on that edge update their opinion to be the average of the two.  Boyd,  Ghosh,    Prabhakar,  and Shah~\cite{boyd2006randomized}  analyzed such models in the settings of gossip algorithms    and proved spectral bounds for convergence times; these bounds are polynomial in $n$ when the edge is chosen uniformly.

Deffuant, Neau, Amblard and Weisbuch \cite{deffuant2000mixing} proposed a model similar to that of Boyd et al except that the vertices update their opinion only if the the distance between the opinions is smaller than a certain fixed parameter. Deffuant et al showed that the opinions converge to several clusters. In each cluster there is consensus, but the clusters are so far in their opinions that they do not exchange any information. 

Kempe, Dobra,   and Gehrke,\cite{kempe2003gossip} 
use gossip algorithms to aggregate information.
Other gossip algorithms were analyzed in  \cite{hedetniemi1988survey,kempe2002protocols}.  
In the voter model, the opinions are in $\{0,1\}$. In each step, a random vertex selects a random neighbour and adopts its opinion.  See, e.g., \cite{sood2005voter,cox1986diffusive,castellano2003incomplete,holley1975ergodic}. Chazelle~\cite{chazelle} relates opinion dynamics to a fragmentation process.

%
\section{Convergence Rate on Finite Graphs} \label{sec:finite}
The goal of this section is to estimate the rate of convergence of asynchronous DeGroot dynamics from arbitrary initial opinions to consensus on finite graphs.


Every irreducible  transition matrix $P$  has a unique stationary distribution $\pi$: a row vector satisfying $\pi P=\pi$. The chain is called \emph{reversible} if $$\forall  v,w \in V, \quad \pi(v) P_{vw}=\pi(w)P_{wv} \,.$$ 
In that case, all its eigenvalues are in $[-1,1]$. If $1=\lambda_1>\lambda_2$ are the two top eigenvalues of $P$, then $\gamma:=1-\lambda_2$ is called the {\em spectral gap\/} of $P$.
Let $\pi_{\min}$ denote the minimal element of $\pi$. If $P$ corresponds to a simple random walk (SRW) on a graph $G=(V,E)$, then $\pi_{\min} \ge \frac1{2|E|}$.
          
For any function $f:V \to {\mathbb R}$ we define its {\em oscillation}   
by $${\rm osc}(f):=\max_{v \in V} f(v) -\min_{w \in V} f(w)\,.$$
The next theorem bounds the convergence time to consensus for reversible chains and SRW on graphs. 
\begin{thm}\label{rate}
Let $P$ be reversible and let $\gamma$ denote its spectral gap. Suppose  ${\rm osc}(f_0) \le 1$. Consider the stopping time 
$$\tau_\eps=\tau_\eps (f_0):=\min\{ t \ge 0: \,{\rm osc}(f_t) \le \eps \} \,.
$$
Then, for $0<\eps <1$, we have
\begin{itemize} 
\item[(a)] $\displaystyle {\mathbb E}(\tau_\eps ) \le  \frac{1}{\gamma}\log\Bigl(\frac{2e}{\gamma  \cdot \pi_{\min} \eps^2}\Bigr)  $.
\item[(b)] If $P$ represents a simple random walk on the graph $G=(V,E)$, then   
\begin{equation}
 \displaystyle {\mathbb E}(\tau_\eps ) \le 8{\rm diam}(G) \cdot |E| \cdot \lceil \log_2(1/\eps ) \rceil \,.   
\end{equation}
\end{itemize}
\end{thm}
We prove this theorem in a sharper form in the next two subsections.  In Section \ref{nonr}, we bound the convergence time for irreducible chains that need not be reversible.

\noindent{\bf Remarks.}
\begin{enumerate}
\item This theorem implies part (i) of Theorem \ref{cor:yuval}.
To obtain the last sentence of that theorem, we use the known fact  that the diameter of a regular graph with $n$ nodes and degree $\Delta$ is at most $3n/\Delta$; see, e.g., the proof of Proposition 10.16 in \cite{Levin-Peres}.
\item Note that if $ f_t(v)\in [a,a+\eps ]$ for all $v \in V$, then the limiting consensus $f_\infty$ will also be in this interval.
\item On the $n$-cycle, we have $\gamma =(2+o(1))\pi^2/n^2$ (See Section 12.3.1 in \cite{Levin-Peres}), so the bound in (b) is better; but in most graphs, (a) is superior. For example, on expander graphs, where $\gamma$ is bounded away from 0, the bound in (a) is logarithmic in $n$, while the bound in (b) is larger than $n$. 
\item By taking  the initial opinion vector $f_0$ to be the right eigenvector of $P$ corresponding to $\lambda_2$, we show in Section 5 that
${\mathbb E}(\tau_\eps )$ is at least of order $\gamma^{-1}$ for constant $\eps $, so the bound in (a) is sharp up to the logarithmic factor.
\end{enumerate}


\medskip
\subsection{Preliminaries}
We first recall some  notation and results from Markov chain theory, that can be found, e.g., in \cite{Saloff} and  \cite{Levin-Peres}.  
\begin{itemize}
\item We will use the $\pi$-weighted scalar product on $\RR^V$:
$$\langle f,g\rangle=\langle f,g\rangle_{\pi}:=\sum_{v\in V}\pi(v)f(v)g(v)$$
and the corresponding norm $\|f\|=\|f\|_\pi:=\sqrt{\langle f,f\rangle}$.
\item The \emph{time reversal} of a Markov chain with transition matrix $P$ and stationary distribution $\pi$, is given by the matrix $P^*$, where
$$\forall  v,w \in V, \quad \pi(v) P^*_{vw}=\pi(w)P_{wv} \,,$$ 
so $P^*$ is the adjoint of $P$ with respect to the $\pi$-weighted scalar product.
The chain is   reversible iff $P^*=P$.
In general, since $P$ is irreducible,  $P^*$ is also irreducible
and has the same stationary distribution $\pi$. This also holds for the \emph{symmetrization}
$ \widetilde{P}:=\frac{P+P^*}{2} \,,$
which is reversible.
\item The \emph{empirical mean} of $f:V \to \RR$ is
$$E_\pi(f):=  \langle f, 1 \rangle_\pi  \,.
$$
\item The \emph{empirical variance} of $f:V \to \RR$ is
$$\vpi(f):=\| f\|_\pi^2 - \langle f, 1 \rangle_\pi^2 \,.
$$
\item The {\em energy} of $f:V\rightarrow \mathbb{R}$ is \begin{equation} \label{defenergy}
\mathcal{E}_P(f):= \langle(I-P)f,f\rangle= \frac{1}{2}\sum_{v,w \in V} \pi(v)P_{vw}|f(v)-f(w)|^2.
\end{equation}
(See, e.g., Lemma 13.6 in \cite{Levin-Peres} for this identity in the reversible case, and \cite{Fill91} or \cite{Saloff} for the general case.)
Note that the time reversal $P^*$ and the symmetrization $\wP$ have the same energy functionals as $P$:
\begin{equation} \label{energy}
\forall f:V \to \RR, \quad \mathcal{E}_P(f)=\mathcal{E}_{P^*}(f)=
\mathcal{E}_{\wP}(f) \,.
\end{equation}
Denote by $\lambda_2(\wP)$ the second largest eigenvalue of the  reversible matrix $\wP$, and let $\gamma:=1- \lambda_2(\wP)$ denote its \emph{spectral gap}.
\item There is a  variational formula for the spectral gap: 
\begin{equation} \label{poinc}
    \gamma=\min\left\{\frac{\cE_P(f)}{\vpi(f)}\, | \, f:V\rightarrow \mathbb{R} \text{ not constant} \right\},
    \end{equation}
    See, e.g., eq.\ (1.2.13) in \cite{Saloff} or Remark 13.8 in \cite{Levin-Peres}.
\end{itemize}

\noindent{\bf Definition.}
Let $ \cF_t  $ denote the $\sigma$-algebra generated by    $\{f_s\}_{s \le t}$. In this section $f_0$ is deterministic, so $ \cF_t  $ is generated by the clock rings in $[0,t]$. Let $\Psi:[0,\infty) \times\RR^V\to\RR$ be continuous. If the limit
\begin{equation}
D\bigl(\Psi(t,f_t)\bigr):=\lim_{h \to 0} \frac1h \EE\bigl[\Psi\bigl(t+h,f_{t+h}\bigr)- \Psi\bigl(t,f_{t}\bigr) | \cF_t\bigr]
\end{equation}
exists, we refer to it as the \emph{drift} of the process $\{\Psi(t,f_t)\}$ at time $t$. In Markov process theory (see, e.g., \cite[Chap. 6]{LeGall}), $D$ is  the infinitesimal generator of the time-space process $\{(t,f_t)\}_{t \ge 0}$. In particular, if  
$D\bigl(\Psi(t,f_t)\bigr) \equiv 0$ for all $t \ge 0$, then $\{\Psi(t,f_t)\}_{t \ge 0}$ is a martingale, and if $D\bigl(\Psi(t,f_t)\bigr) \le  0$ for all $t \ge 0$, then $\{\Psi(t,f_t)\}_{t \ge 0}$ is a supermartingale.

\medskip

The drift operator $D$ satisfies a version of the product rule for derivatives. 
Suppose that   $\varphi:[0,\infty) \to \RR$ is differentiable and $D\bigl(\Psi(t,f_t)\bigr)$ exists for all $t$.
Then
\begin{equation}
    \begin{split}
        &\varphi(t+h) \cdot \Psi(t+h,f_{t+h})-\varphi(t) \cdot \Psi(t,f_t)\\ &=\varphi(t+h) \cdot \Bigl(\Psi(t+h,f_{t+h})-  \Psi(t,f_t)\Bigr)+\bigl(\varphi(t+h)-\varphi(t)\bigr) \cdot \Psi(t,f_t) \,.
    \end{split}
\end{equation}
Taking conditional expectation given $\cF_t$, dividing by $h$,  and  letting $h \to 0$ gives
\begin{equation} \label{prod} D\bigl(\varphi(t) \cdot \Psi(t,f_t)\bigr)=\varphi(t) \cdot D\bigl(\Psi(t,f_t)\bigr)+ \varphi'(t)\cdot \Psi(t,f_t)  \,. \end{equation}
In the next lemma, we record the drifts of some key summaries of the opinion profile $f_t$.
\begin{lem} \label{lemdrift}
The empirical mean $M_t:=E_\pi(f_t)$ is a martingale, and the drift of $M_t^2$ satisfies
\begin{equation} \label{M2drift}
 \pi_{\min} \|(I-P)f_t\|_\pi^2 \le D(M_t^2) \le \pi_{\max} \|(I-P)f_t\|_\pi^2\,.
\end{equation}
We also have
\begin{equation} \label{Dsquare}
  D(\|f_t\|^2 )
  =\langle(P^*P-I)f_t,f_t\rangle=-\cE_{P^*P}(f_t) \,.
\end{equation}
\end{lem}
\begin{proof}
We have
$$\EE[M_{t+h}-M_t | \cF_t]=h\sum_{v \in V} \pi(v) \bigl((Pf_t)(v)-f_t(v)\bigr)+O(h^2)=h\pi(P-I)f_t+O(h^2)=O(h^2) \,.$$
Dividing by $h$ and letting $ h \to 0$, we see that $\{M_t\}$ has zero drift, so it is a martingale. 

Next, we consider the second moment of $M_t$. By orthogonality of Martingale increments,   
$$\EE[M_{t+h}^2-M_t^2  | \cF_t]=\EE[(M_{t+h}-M_t)^2 | \cF_t]=h\sum_{v \in V} \pi(v)^2 \bigl((Pf_t)(v)-f_t(v)\bigr)^2+O(h^2) \,.$$
Dividing by $h$ and passing to the limit yields 
\begin{equation} \label{M2true}
    D(M_t^2)=\sum_{v \in V} \pi(v)^2 \bigl((Pf_t)(v)-f_t(v)\bigr)^2  \,,
\end{equation}
which implies \eqref{M2drift}. Finally,
$$\EE \big[ \|f_{t+h}\|^2-\|f_t\|^2 | \cF_t \big] =h\sum_{v\in V} \pi(v)\bigl((Pf_t)(v)^2-f(v)^2\bigr)+O(h^2) \,.
$$
Thus $D(\|f_t\|^2 )=\|Pf_t\|^2-\|f_t\|^2$, and  \eqref{Dsquare} follows.
\end{proof}

\subsection{Bounds for reversible chains}
In this section we prove Part (a) of Theorem~\ref{rate} which immediately follows from the following proposition.
\begin{prop}
If $P$ is reversible, then the time to consensus satisfies 
$$\EE(\tau_\eps)\le  t_\eps:=\frac{1}{\gamma}\log\Bigl(\frac{4e\cE_P(f_0)}{\gamma  \cdot \pi_{\min} \eps^2}\Bigr) \,.$$
Moreover,
$$\forall t, \quad \PP(\tau_\eps>t_\eps+t) \le e^{-(\gamma t+1)}      \,.$$
   
\end{prop}
 \begin{proof} If $f:V\rightarrow \mathbb{R}$ is updated at $v$ according to our dynamics, then the energy $\mathcal{E}_P(f)$ is decreased by 
\begin{eqnarray} \nonumber  && \sum_{w\in V} {\pi(v)P_{vw}} \Bigl(\bigl(f(v)-f(w)\bigr)^2- \bigl(Pf(v)-f(w)\bigr)^2\Bigr) \\  \nonumber &=&
 \pi(v) \sum_{w\in V} P_{vw} \bigl(f(v)-Pf(v)\bigr) \cdot     \bigl(f(v)+Pf(v)-2f(w)\bigr)  \\  \nonumber &=& \pi(v)\bigl((I-P)f\,(v)\bigr)^2 \,.
\end{eqnarray}
(This computation used reversibility; in the non-reversible case, the coefficient $\pi(v)P_{vw}$ in the first line would be replaced by its average with $\pi(w)P_{wv}$. In that  case, the  energy could increase in an update and need not be a supermartingale, see the first paragraph of Subsection~\ref{nonr}.)
Abbreviating  $\cE_t:=\cE_P(f_t)$, we infer from the preceding display that 
\begin{equation} \label{precauchy} \mathbb{E}\left[\mathcal{E}_t-\mathcal{E}_{t+h}|{\mathcal F}_t\right]=h\sum_{v\in V}\pi(v)\left[(I-P)f_t(v)\right]^2 +O(h^2) =h\|(I-P)f_t\|^2+O(h^2)\,.
\end{equation}
Dividing by $h$ and letting $h \downarrow 0$, we get
\begin{equation} \label{energydrift}  
D(\cE_t)=-\|(I-P)f_t\|^2 \,.
\end{equation}

 By the variational formula \eqref{poinc}, we have
$$ \gamma\|f_t-M_t\|^2 =\gamma \cdot \vpi(f_t) \le \cE_t \,.$$
Observe that $\langle (I-P)f,1 \rangle=0$ for any $f : V\to {\mathbb R}$, so    by Cauchy-Schwarz,
\begin{eqnarray} \label{cauchy2} \nonumber
  \|(I-P)(f_t)\|^2\cdot  \cE_t &=&\|(I-P)(f_t-M_t)\|^2  \cdot   \cE_t \\ \nonumber &\ge& \|(I-P)(f_t-M_t)\|^2\cdot \gamma  \|f_t-M_t\|^2 \\ &\geq&  \gamma \langle(I-P)(f_t-M_t),(f_t-M_t)\rangle^2=\gamma \cE_t^2 \,.
\end{eqnarray}
Therefore,
\begin{equation} \label{cauchy3}
 \|(I-P)(f_t)\|^2    \geq  \gamma \mathcal{E}_t \,.
\end{equation}
 
Inserting this in \eqref{energydrift} yields
\begin{equation} \label{precauchy2}
D(\cE_t)  \le  -\gamma  \cE_t \,.
\end{equation}
By the product rule,
$$\Gamma_t:=e^{\gamma t} \cE_t $$ satisfies
$D(\Gamma_t) \le 0$, so $\{\Gamma_t\}$ is a non-negative supermartingale.
Observe that
$$\text{osc}(f_t) \ge \eps \Rightarrow \: \exists v :\,  |f_t(v) - M_t|  \ge \eps/2 \Rightarrow \:  \vpi(f_t) \ge \frac{\pi_{\min} \eps^2}4 \Rightarrow \:  \cE_t \ge  \frac{\gamma  \cdot \pi_{\min} \eps^2}4 \,.
$$
Consequently,
$$\cE_0=\Gamma_0 \ge \EE(\Gamma_t) \ge e^{\gamma t} \PP(\tau_\eps>t) \frac{\gamma  \cdot \pi_{\min} \eps^2}4 \,,
$$
so $\PP(\tau_\eps>t) \le e^{-\gamma t} \frac{4\cE_0}{\gamma  \cdot \pi_{\min} \eps^2}$.
Setting
$$t_\eps^*:=\frac{1}{\gamma}\log\Bigl(\frac{4\cE_0}{\gamma  \cdot \pi_{\min} \eps^2}\Bigr) \,,$$
we obtain that $\PP(\tau_\eps>t_\eps^*+t) \le e^{-\gamma t}$ for all $t$, so $\EE(\tau_\eps) \le t_\eps=t_\eps^*+1/\gamma$.
\end{proof} 
The next corollary bounds the variance of the consensus $f_\infty$.  
\begin{cor} \label{thm:varcons1}
  If $P$ is reversible, then 
  $$\pi_{\min} \cE_P(f_0)      \le \var(f_\infty) \le \pi_{\max} \cE_P(f_0) \,.$$
\end{cor}
\begin{proof}
Observe that $M_\infty=\lim_t M_t=f_\infty$ a.s. We may assume that $M_0=E_\pi(f_0)=0$, since subtracting a constant from $f_0$ does not affect  
$\cE_P(f_0)$ and $\var(f_\infty)$. By Lemma \ref{lemdrift} and \eqref{energydrift},
$D\bigl(M_t^2+\pi_{\max}\cE_P(f_t)\bigr) \le 0$, so
$ M_t^2+\pi_{\max}\cE_P(f_t) $ is a supermartingale.
Thus, by  Lebesgue's bounded convergence theorem,
$$  \pi_{\max}\cE_P(f_0) \ge \EE(M_\infty^2)=\var(M_\infty)=\var(f_\infty) \,.$$
This proves the upper bound on $\var(f_\infty)$; the lower bound is proved similarly, using the submartingale $ M_t^2+\pi_{\min}\cE_P(f_t) $. 
\end{proof}
{\noindent \bf Remark.} The corollary above assumed that the initial opinion profile $f_0$ is deterministic. If $f_0$ is random, then  the same argument gives
$$\pi_{\min} \mathbb E \big[\cE_P(f_0) \big]      \le \var(f_\infty)-\var(M_0) \le \pi_{\max}  \mathbb E \big[\cE_P(f_0) \big] \,.$$
\subsection{The case of simple random walk}
In this section we prove Part (b) of Theorem~\ref{rate}. We begin by presenting a few well known notions and results regarding simple random walks. Since we assume each edge in $E$ has unit resistance, the  {\bf effective conductance} ${\mathcal C}(v\leftrightarrow w)$ from $v$ to $w$ is given by 
\[{\mathcal C} (v\leftrightarrow w):=2|E|\cdot \min\left\{\mathcal{E}(f)\, |\,  f:V\rightarrow \mathbb{R}, f(v)=1, f(w)=0\right\}\, \]
  (See, e.g., the solved exercise 2.13 in \cite{Lyons-Peres}.) The  {\bf effective resistance} is ${\mathcal R}(v\leftrightarrow w):=\frac{1}{{\mathcal C}(v\leftrightarrow w)}$.

\begin{itemize}
    \item The \emph{commute time} for simple random walk between $v$ and $w$ equals $2|E|R(v\leftrightarrow w)$ (See Prop.\ 10.7 in \cite{Levin-Peres}.)
    \item Write ${\mathcal R}_{\rm max}:=\max_{v,w\in V} {\mathcal R} (v \leftrightarrow w)$.
    Then ${\mathcal R}_{\rm max} \le {\rm diam}(G)$, and (as we show in the proof below) in Part (b) of Theorem \ref{rate}, ${\rm diam}(G)$ can be replaced by ${\mathcal R}_{\rm max}$.
    
\end{itemize}

\noindent{\bf Further Remarks:}
\begin{enumerate}
\item[1.] For a simple random walk on $G=(V,E)$ where $|V|=n$, we have $\gamma\geq \frac{c}{n|E|}$.  
\item[2.] On a regular graph, $\gamma\geq \frac{c}{n^2}$ and ${\rm diam(G)}\leq \frac{3n}{\Delta}$ where $\Delta=\max_v d_v$. In all these inequalities, $c>0$ is a universal constant.
\item[3.] On the hypercube with $n=2^\ell$ nodes, $\gamma=\frac{2}{\ell}$ but ${\mathcal R}_{\rm max} \in [1,3]$.  
\end{enumerate}
 

\begin{proof}[Proof of Theorem \ref{rate} (b)]
 Since ${\rm osc}(f_0)\leq 1$, by subtracting a constant from $f_0$, we may assume WLOG that $\max_{v \in V} |f_0(v)| \le 1/2$,
whence $\max_{v \in V} |f_t(v)| \le 1/2$ for all $t \ge 0$.
Recall that $\cE_t:=\cE_P(f_t)$ and, by \eqref{energydrift}, $ D(\cE_t)=-\|(I-P)f_t\|^2 $.
Since $\|f_t\| \leq 1/2$, we get via Cauchy-Schwarz that
\begin{equation} \label{cauchy}
\|(I-P)f_t\|^2\geq 4\|(I-P)f_t\|^2\cdot \|f_t\|^2 \geq 4\langle(I-P)f_t,f_t\rangle^2=4\cE_t^2 \,.
\end{equation}
Therefore, by \eqref{energydrift}  we have 
\begin{equation} \label{eq:enerdrift} D(\cE_t)=-\|(I-P)f_t\|^2\leq -4\cE_t^2 \,.
\end{equation}
Let $L_t:=4\mathbb E[\cE_t]$. 
We claim that 
\begin{equation}\label{equation L_t}
 L_t\leq \frac{1}{t+L_0^{-1}},\quad  \forall t\geq 0.    
\end{equation}
Let $N(t)$ be the total number of clock rings in $[0,t]$. To verify the claim, we first observe that for each $k \ge 0$, the conditional expectation
$a_k:=\mathbb E[\cE_t|N(t)=k]$ does not depend  on $t$, and satisfies $|a_k| \le 1$.
Therefore,
$$L_t=4\sum_{k=0}^\infty a_k \, {\mathbb P} (N(t)=k)=4\sum_{k=0}^\infty a_k \frac{(nt)^k}{k!}e^{-nt}$$
is a smooth function of $t$.
Since the energy is nonincreasing,  Fatou's Lemma gives 
$$-L'_t=4\lim_{h \downarrow 0} {\mathbb E} \Bigl[\frac{ \cE_t-\cE_{t+h}}{h} \Bigr] =4\lim_{h \downarrow 0} {\mathbb E} \Bigl( {\mathbb E}\Bigl[\frac{ \cE_t-\cE_{t+h}}{h} \Big| {\mathcal F}_t\Bigr] \Bigr)\geq -4\mathbb E[D(\cE_t)]\,.$$ By \eqref{eq:enerdrift}  and   Jensen's inequality, 
$$-L'_t \ge 16 \mathbb E[\cE_t^2]  \ge L_t^2 \,,
$$ 
whence
$ (1/L_t)'= {-L_t'}/{L_t^2} \ge 1 \,.$ 
Therefore $1/L_t\ge t+1/L_0$ and \eqref{equation L_t}
is verified.

Now, by the definition of $\mathcal R_{\mathrm{max}}$,
 $$ \tau_{1/2}>t \Rightarrow \: \text{osc}(f_t) \ge 1/2 \Rightarrow \: 4\cE_t \ge \frac{1}{2|E| {\mathcal R}_{\rm max} }  \,.
$$
Therefore, setting $t_*=4{\mathcal R}_{\rm max}|E|$ we get 
$$
\PP(\tau_{1/2} > t_*) \le \mathbb P \Big( 4\cE_{t_*} \ge \frac{1}{2|E|{\mathcal R}_{\rm max}} \Big) \leq L_{t_*}\cdot 2|E|{\mathcal R}_{\rm max} \leq \frac 12.
$$
Where the second inequality is Markov's inequality, and the third inequality follows from \eqref{equation L_t}. The same argument shows that 
$$
\PP(\tau_{1/2} > kt_*\mid\tau_{1/2}>(k-1)t_*)\leq \frac 1 2, \quad \forall k\geq 1.
$$
It follows that 
\begin{equation}\label{quad}
\mathbb E[\tau_{1/2}]\leq 2t_*=8{\mathcal R}_{\rm max}|E|.    
\end{equation}
Next, we use an iteration argument and  argue inductively that for all $k \ge 1$ and all initial opinion profiles $f_0$, we have
\begin{equation} \label{iter}
\mathbb{E}\tau_{2^{-k}} \le 8k{\mathcal R}_{\rm max}|E|.  
\end{equation}
Indeed, the base case $k=1$ is just \eqref{quad}.
Next, observe that 
$$\tau_{2^{-k-1}}(f_0) \le \tau_{1/2}(f_0)+
\tau_{2^{-k}}(2f_{\tau_{1/2}})\,,
$$
so applying the induction hypothesis and additivity of expectation proves \eqref{iter} for all $k \ge 1$. Taking $k=\lceil \log_2(1/\eps ) \rceil$ completes the proof of (b).

\medskip
 
\end{proof}
\subsection{Beyond reversible chains} \label{nonr}
In digraphs, the energy can increase and need not be a supermartingale. A simple  example is the directed walk on the $n$-cycle $\{0,1,\ldots, n-1\}$, where $P_{ij}=1$ iff $j=i+1 \mod n$.
If the initial configuration $f_0$ is given by  $f_0(k)=\frac1n \min\{k,n-k\}$, and $n$ is even, then $$\EE\bigl(\cE(f_h)-\cE(f_0)\bigr)=2h(n-4)/n^2+O(h^2)\,.$$ This issue persists if we make $P$ lazy by averaging it with the identity, or if we change $P$ to satisfy
$P_{ij}=1/2$ iff $j-i \in \{1,2\} \mod n$.
Instead of tracking the energy, we will track the empirical variance.

  In the next theorem, the matrix $P^*P$ will be important. This matrix could be reducible even if $P$ is irreducible (E.g., for the directed walk on the cycle,   $P^*P=I$.) Recall that $\wP=(P+P^*)/2$ and $\gamma=1-\lambda_2(\wP)$. Let $\gamma_{P^*P}=1-\lambda_2(P^*P)$. Note that $\gamma_{P^*P}>0$ iff $P^*P$ is irreducible.
\begin{thm} \label{genrate}
 Suppose that $4\vpi(f_0) \ge \eps^2 \pi_{\min}$. (Otherwise  $\, \tau_\eps=0$, see \eqref{recall}). Then

    \medskip
    
\begin{itemize}
    \item[{\rm (i)}] $\EE(\tau_\eps)\le \displaystyle \frac{1}{\widehat{\gamma}} \log\Bigl(\frac{4e \cdot \vpi(f_0)}{\eps^2 \pi_{\min}}\Bigr)$, 
    where $$\widehat{\gamma}:=
(1-\pi_{\min})\gamma_{P^*P}+2\pi_{\min}\gamma \,.
$$
    
    \medskip
    
    \item[{\rm (ii)}] If $P_{vv} \ge \delta>0$ for all $v \in V$,
    then   
  $$\EE(\tau_\eps)\le \displaystyle \frac{1}{2 \delta \gamma} \log\Bigl(\frac{4e \cdot \vpi(f_0)}{\eps^2 \pi_{\min}}\Bigr)\,.$$
\end{itemize}
\end{thm}
\begin{proof}
\noindent{\bf (i)}
By Lemma \ref{lemdrift}, the empirical variance $\vpi(f_t)=\|f_t\|^2-M_t^2$ satisfies  
\begin{equation} \label{vardrift}
 D(\vpi(f_t))\le -\cE_{P^*P}(f_t)- \pi_{\min} \|(I-P)f_t\|_\pi^2  \,.
\end{equation}
For any $f:V \to \RR$,
\begin{equation} \label{eq:kid}
\|(I-P)f\|_\pi^2 =
\langle (I-P^*)(I-P)f,f \rangle=
\langle (I-P^*-P+P^*P)f,f\rangle=
2\cE_P(f)-\cE_{P^*P}(f) \,,
\end{equation}
so 
\begin{equation} \label{vardrift2}
 D(\vpi(f_t))\le (\pi_{\min}-1)\cE_{P^*P}(f_t)- 2\pi_{\min}  \cE_P(f_t) \,.
\end{equation}

Using the variational principle twice yields
\begin{equation} \label{vardrift3}
 D(\vpi(f_t))\le (\pi_{\min}-1)\cE_{P^*P}(f_t)- 2\pi_{\min} \cE_P(f) \le -\widehat{\gamma} \vpi(f_t) \,.
\end{equation}
By the product rule,
$\Lambda_t:=e^{\widehat{\gamma}t} \vpi(f_t)$ satisfies
$D(\Lambda_t) \le 0$, so $\Lambda_t$ is a supermartingale. Recall that
\begin{equation} \label{recall}
\text{osc}(f_t) \ge \eps \Rightarrow \: \exists v :\,  |f_t(v) - M_t|  \ge \eps/2 \Rightarrow \:  \vpi(f_t) \ge \frac{\pi_{\min} \eps^2}4 \,.
\end{equation}
Therefore,
$$\vpi(f_0)=\Lambda_0 \ge \EE(\Lambda_t) \ge e^{\widehat{\gamma}t} \PP(\tau_\eps>t) \frac{\pi_{\min} \eps^2}4 \,.
$$
Letting $$\hat{t}_\eps:=\frac{1}{\widehat{\gamma}}
\log \Bigr(\frac{4\cdot \vpi(f_0)}{\pi_{\min} \eps^2}\Bigl) \,,
$$
we deduce that
$$\PP(\tau_\eps>\hat{t}_\eps+t) \le e^{-\widehat{\gamma}t} \,,
$$
so 
$\EE(\tau_\eps) \le  \hat{t}_\eps+1/\widehat{\gamma}$.

\medskip

\noindent{\bf (ii)} The hypothesis that $P_{vv} \ge \delta$ for all $v$ means that $P=\delta I+(1-\delta)Q$ for some stochastic matrix $Q$. Then
$$I-P^*P=I-\bigl(\delta I+(1-\delta)Q^*\bigr)\bigl(\delta I+(1-\delta)Q\bigr)=(1-\delta)^2(I-Q^*Q)+\delta(1-\delta)\bigl(2I-Q-Q^*\bigr)\,.
$$
Therefore, for any $f:V \to \RR$, we have
\begin{equation} \label{energycalc}
  \cE_{P^*P}(f) \ge 2\delta(1-\delta)\cE_{Q}(f)=2\delta \cE_{P}(f)\,.
\end{equation}
The variational principle yields that
$\gamma_{P^*P} \ge  2\delta \gamma\,,$
whence $\widehat{\gamma} \ge 2\delta\gamma$, so the claim follows from (i).
\end{proof}
Recall that a digraph is Eulerian iff it is connected,  and for every vertex its out-degree equals its indegree.
\begin{cor} \label{cor:euler}
Let $P$ be the transition probabilities matrix of a simple random walk on an Eulerian digraph $G=(V,E)$ with $n$ vertices and $m$ edges. I.e., $P_{vw}=\text{deg}(v)^{-1}{\bf 1}_{\{(v,w) \in E\}}$. If ${\rm osc}(f_0) \le 1$, then
$$\EE(\tau_\eps)\le \displaystyle Cnm^2 \log\Bigl(\frac{e \cdot m}{\eps^2}\Bigr) \,,$$ 
where $C$ is an absolute constant. For a lazy random walk on $G$ (that can be obtained from $P$ by replacing it with $(P+I)/2$), the bound is
$$\EE(\tau_\eps)\le \displaystyle Cnm \log\Bigl(\frac{e \cdot m}{\eps^2}\Bigr) \,.$$ 
\end{cor}
\begin{proof}
Since $G$ is Eulerian, the stationary measure is given by $\pi(v)={\rm deg}(v)/m$. The time-reversal of $P$ is SRW on the $G$ with all edges reversed, and the symmetrization $\wP$ is SRW on the undirected graph obtained by ignoring the orientations of the edges. Thus $1/\gamma\le Cnm$, see, e.g., \cite{Lyons-Gharan} or \cite[Chap. 10]{Levin-Peres}. The claim now follows from Theorem \ref{genrate}.
\end{proof}

Next, we bound  the variance of the consensus $f_\infty$.  In the example of the directed cycle from the beginning of this section, if $f_0(k)=0$ for $0 \le k < n/2$ and $f_0(k)=1$ for $n/2 \le k <n$,
then $\var(f_\infty)=1/4$, and the consensus is far from the initial empirical average. The following corollary ensures this does not occur if $\pi_{\max}$ is small and every node puts a substantial  weight on its own opinion.
\begin{cor}\label{cor:varconsgen}
  If $P_{vv} \ge \delta>0$ for all $v \in V$,
    then $$\var(f_\infty) \le
    \frac{\pi_{\max}\vpi(f_0)}{\delta} \,,$$
\end{cor}
\begin{proof}
Without loss of generality, by subtracting a constant from $f_0$, we may assume that $M_0=0$.
By Lemma \ref{lemdrift} and \eqref{eq:kid},
$$
D(M_t^2) \le \pi_{\max} \|(I-P)f_t\|^2 \le 2\pi_{\max} \cE_P(f_t) \,.
$$
On the other hand, \eqref{vardrift} and \eqref{energycalc} imply that
\begin{equation} \label{vardrift4}
 D(\vpi(f_t))\le -\cE_{P^*P}(f_t) \le -2\delta \cE_{P}(f_t) \,.
\end{equation}
Therefore
$$A_t:=\delta M_t^2+\pi_{\max}\vpi(f_t) $$
is a bounded supermartingale tending in $L^1$  to a limit $A_\infty$,
whence
$$\pi_{\max}\vpi(f_0)=A_0 \ge \EE(A_\infty)
=\delta \EE(M_\infty^2)=\delta\var(f_\infty) \,.
$$
\end{proof}

\section{Fragmentation Process and Related Results}\label{sec: fragmentation}
In this section we introduce a process which is closely related to the DeGroot dynamics. 
Let $X(t)$ be a continuous time Markov chain on $V$ with independent unit Poisson clocks on the vertices and transition probabilities matrix $P$. Let $\mathcal F_t$ be the $\sigma$-algebra generated by the clock rings up to time $t$. The \emph{fragmentation process originating at $o\in V$}, $\{m_t(o,v)\in[0,1]\}_{v\in V,t\geq 0}$, is a left-continuous step process defined by
\begin{equation}\label{def:frag}
    m_t(o,v):=\mathbb{P}\big(X(t)=v \mid  \mathcal F_t,X(0)=o \big).
\end{equation}
Note that the $\{m_t(o,\cdot)\}_{o\in V}$ are all interdependent since they are all measurable w.r.t.\ the same clock rings. 

We will often abbreviate $m_t(v):=m_t(o,v)$ when $o$ is clear from the context or when we make a statement that refers to every $o\in V$. The process $\{m_t(v)\}_{v\in V,t\geq 0}$ is a step process adapted to the filtration $\{\mathcal F_t\}_{t\in\mathbb R_+}$ that satisfies the following Markovian rules:
\begin{itemize}
    \item At time 0, \[
m_0(v):=\begin{cases}
1& v=o,\\
0&v\neq o.
\end{cases}
\]
\item Suppose that the clock of vertex $v$ rings at time $t\geq 0$, then for every $u\in V$,
\[
m_{t}(u):= \mathbf{1}_{\{u\neq v\}}m_{t-}(u)+ P_{vu}m_{t}(v) .
\]
Namely, the mass of vertex $v$ is pushed to its neighbors proportionally to $v$'s row in $P$.
\end{itemize}

We now explain the relation between the DeGroot and the fragmentation processes. We first demonstrate this relation on finite $V$ and then use this relation to define the DeGroot process on infinite $V$. Suppose first that $V$ is finite. In this case, the DeGroot dynamics can be defined through the DeGroot updating rule \eqref{eq:degroot}. In what follows we identify a Poisson process with its cumulative function $N(s)$ that counts the number of rings in the time interval $[0,s]$. Suppose that the Degroot dynamics uses the Poisson processes $N_v(s)$ for $v\in V$. Let $t>0$ and define new Poisson processes $N_v^t(s):=N_v(t)-N_v(t-s)$ for $s\le t$. Consider the the fragmentation processes $m_s^{t}(o,v)$ generated from the Poisson processes $N_v^t(s)$. We have the following identity 
\begin{equation}\label{eq:m_s^t}
    f_t(o)=\sum_{v\in V} m_s^t(o,v)f_{t-s}(v).
\end{equation}
This identity clearly holds for $s=0$ and one can prove it using induction on the times of clock rings $0=s_0<s_1<s_2<\cdots$. In particular,
\begin{equation}\label{eq:degroot fragmentation}
    f_t(o)=\sum_{v\in V} m_t^t(o,v)f_{0}(v).
\end{equation}
We use Equation~\eqref{eq:degroot fragmentation} to define the DeGroot process in general ($V$ possibly infinite). 
\begin{definition}\label{def: degroot}
Let $f_0\sim \mu_0$ independently of the Poisson clocks. The DeGroot dynamics $\{f_t(v)\}_{v\in V, t\geq 0}$ is the process defined by Equation~\eqref{eq:degroot fragmentation}.
\end{definition}





We obtain several results regarding the fragmentation process. In our results,  we assume that the transition probabilities decay at least as fast the inverse square root of the time. 

\begin{assumption}[square root decay of transition probabilities]\label{assumption 1/sqrt t}
We say that a network with matrix  $P$ satisfies assumption 1 with parameter $C_0 \in (0, |V|/2)$, if for every  $k\le |V|^2$ (every  $k\in \mathbb N$ when $V$ is infinite) we have $\sup_{uv} P^k_{uv}\leq C_0 k^{-1/2}$.
\end{assumption}

\begin{remark}\label{remark:assumption}
  For a family of finite networks to satisfy this assumption, the parameter $C_0$ must be the same for all networks in the family.  

There are many interesting networks that satisfy Assumption~\ref{assumption 1/sqrt t}. One example is simple random walks on connected bounded degree, finite or infinite graphs. See, e.g., \cite[Lemma~B.0.2]{amir2021robust} or \cite[Lemma~3.4]{lyons2005asymptotic}. In fact, these lemmas apply to a more general class of reversible Markov chains.

Another class of Markov chains that satisfy Assumption~\ref{assumption 1/sqrt t} is lazy random walks on directed finite Eulerian graphs that are either regular, or have bounded degree. See, e.g., \cite[Lemma 2.4]{BPS18}.
\end{remark}

On infinite networks that satisfy Assumption~\ref{assumption 1/sqrt t}, the fragmentation process converges to zero uniformly almost surely. Theorem~\ref{thm:1} below shows that after time $t$, it is very unlikely that a fraction of the total mass concentrates on one vertex.

\begin{thm}\label{thm:1}
	There is a universal constant $\beta >0$ with the following property. Suppose that $|V|=\infty$ and Assumption~\ref{assumption 1/sqrt t} holds with parameter $C_0$. Then for every $\eps>0$, there exists $t_0=t_0(C_0,\eps )$,   such that for every vertex $o\in V$ and all $t\ge t_0$, we have
	\begin{equation}
	\mathbb P \big( \exists v, \ m_t(v) \ge \eps  \big) < \exp \big( -t^{\beta } \big) ,
	\end{equation}
    where $\{m_t(\cdot)\}$ is the fragmentation process originating at $o$.
\end{thm}

The exponent $\beta $ we obtain in the proof is far from optimal and in fact we conjecture that the following holds:

\begin{conjecture}
For any $\eps >0$ there exists $c_{\eps }>0$ such that $\ \mathbb P (\exists v, \ m_t(v) \ge \eps  ) < e^{-c_{\eps }t}$.
\end{conjecture}

The next three propositions will be used in our proofs of the theorems. These propositions provide bounds on certain moments of the fragmentation process. The proofs of the Propositions are given in Section~\ref{sec:walks}. 

The first proposition bounds the second moment in infinite networks from above.
\begin{prop}\label{prop:fragmentation infinite-graph second moment}
Let $m_t(v)$ be the fragmentation process originating at an arbitrary vertex. Under Assumption~\ref{assumption 1/sqrt t}, there exists a universal constant $\alpha >0$ and a real number $t_0=t_0(C_0)$ such that 
\[
\mathbb E\left[\sum_{v\in V}(m_t(v))^2\right]\leq t^{-\alpha},
\]
for every $t\in (t_0,|V|^2/2)$, where $V$ may be either finite or infinite. 
\end{prop}

Note that the constant $\alpha$ cannot be greater than $\frac 1 2$, since if $P$ is the transition matrix of a symmetric random walk on the line graph $\mathbb Z$, then \[
\mathbb E\left[\sum_{v\in V}(m_t(v))^2\right]\geq \sum_{v\in V}\left[\mathbb Em_t(v)\right]^2=\sum_{v\in V}\left[P^{(t)}_{ov}\right]^2=\Theta(t^{-1/2}),
\]
where $P^{(t)}$ is the transition probabilities matrix of the continuous time symmetric random walk.

Our proof of Proposition~\ref{prop:fragmentation infinite-graph second moment} provides a certain $0<\alpha<\frac 1 2$. In Remark~\ref{remark: imporving alpha}, we explain why we conjecture that the proposition should hold with $\alpha$ arbitrarily close to $\frac 1 2$. 

The next proposition studies the case of reversible chains. Under this assumption, we obtain the improved (and optimal) bound of $O(t^{-1/2})$.
\begin{prop}\label{prop:fragmentation random-walk second moment}
Let $V$ be a (finite or infinite) network satisfying Assumption~\ref{assumption 1/sqrt t}. Suppose in addition that $P$ is the transition matrix of a reversible Markov chain on $V$ with stationary measure $\pi$. Let $\Delta=\sup\limits_{v,u\in V}\pi_v/\pi_u$ and let $m_t(v)$ be the fragmentation process originating at an arbitrary vertex. Then, there exists a constant $C$ depending on $C_0$ and $\Delta $ such that for all $t\le |V|^2/3$ we have
\[
\mathbb E\left[\sum_{v\in V}(m_t(v))^2\right]\leq  C t^{-1/2}.
\]
\end{prop}

The third proposition refers to higher moments of the fragmentation process.

\begin{prop}\label{prop:fragmentation infinite-graph d-th moment}
Let $0<\alpha <1$ and $C_1>1$. Suppose that $V$ is either finite or infinite and that the network satisfies Assumption~\ref{assumption 1/sqrt t}. In addition, suppose that the network has the property that the fragmentation process originating at any vertex satisfies 
\[
\mathbb E\left[\sum_{v\in V}(m_t(v))^2\right]\leq  C_1 t^{-\alpha }
\]
for all $t\le |V|^2/3$. Then, there exists $t_0=t_0(C_0,C_1,\alpha )$ such that for all $d\ge 2$ and all $\max (t_0,d^{10/\alpha }) \le t \le |V|^2/3$ we have  
\[
\mathbb E\left[\sum_{v\in V}(m_t(v))^d\right]\leq 
t^{-\alpha d/10}.
\]
\end{prop}

Theorem~\ref{thm:1} easily follows from Propositions~\ref{prop:fragmentation infinite-graph second moment} and~\ref{prop:fragmentation infinite-graph d-th moment}. Indeed, let $\alpha $ be the constant from Proposition~\ref{prop:fragmentation infinite-graph second moment}. Let $t>1$ sufficiently large and $d:=t^{\alpha /10}$. By Markov's Inequality, we have
\begin{equation}
\begin{split}
    \mathbb P (\exists v, \ m_t(v) \ge \eps  ) \le \sum _{v\in V} \mathbb P \big( m_t(v)^d \ge \eps ^d \big) &\le \eps ^{-d} \sum _{v\in V} \mathbb E \big[ m_t(v)^d \big]   \\
    &\le \big( \eps ^{-1} t^{-\alpha /10} \big)^d \le \exp \big(-t^{\alpha /10}\big).
\end{split}
\end{equation}

\section{Markov Chains with Shared Clocks}
\label{sec:walks}
In this section, we prove Propositions~\ref{prop:fragmentation infinite-graph second moment}, \ref{prop:fragmentation random-walk second moment} and \ref{prop:fragmentation infinite-graph d-th moment}.
The proofs begin by translating the discussion on the moments of the fragmentation process to a discussion on certain joint Markov chains.

Let $X_1(t),X_2(t),\ldots$ be continuous-time Markov chains on $V$ with the following joint Markovian law. They all use the same Poisson clocks but have independent trajectories (with transition probabilities given by $P$). Thus, each $X_i(t)$ has the same law (as $X(t)$ from Equation~\eqref{def:frag}) and they are correlated with each other only through sharing the same Poisson clocks. 

\begin{claim}\label{claim:1}
For every $d\in\mathbb N$, 
\begin{equation}
    \mathbb E \big[ (m_t(v))^d \big] =\mathbb P \big( X_1(t)=X_2(t)=\cdots =X_d(t)=v \big) .
\end{equation}
\end{claim}

\begin{proof}
Let $\mathcal F$ denote the sigma-algebra generated by the clock rings. By the definition of the fragmentation process,  
\begin{equation}\label{eq: m_v}
m_t(v)=  \mathbb P(X_1(t)=v\ | \ \mathcal F).    
\end{equation}
 
Let $d\in \mathbb N$ and $t\geq 0$. Conditioning on $\mathcal F$, the trajectories of the random walks are independent and identically distributed and therefore 
\begin{equation}
 \mathbb P \big( X_1(t)=\cdots=X_d(t)=v \ |\ \mathcal F \big) =\prod_{j=1}^d\mathbb P \big( X_j(t)=v \ | \ \mathcal F \big) =\mathbb P \big( X_1(t)=v \ | \ \mathcal F \big) ^d= m_t(v)^d.
\end{equation}
Claim~\ref{claim:1} follows by taking expectations.
\end{proof}

The reversible case is easier than the non-reversible case and therefore we begin with the proof of Proposition~\ref{prop:fragmentation random-walk second moment}.

\subsection{Proof of Proposition~\ref{prop:fragmentation random-walk second moment}}

Let $X_1(t)$ and $X_2(t)$ be independent continuous-time Markov chains with shared Poisson clocks. By Claim~\ref{claim:1} it suffices to bound the probability that $X_1(t)=X_2(t)$.

We start with some notations that will be useful throughout the proof. A path $p$ is a finite or infinite sequence of states of the Markov chain. We let $|p|$ be the number of steps taken by the path and let $p^{-1}$ be the reversed path. For an integer $n$ and a path $p$ with $|p| \ge n$ we denote by $p_n$ the path obtained from the first $n$ steps of $p$.

Next, let $\mathcal A _1 (p)$ be the event that the first $|p|$ steps of $X_1$ were along the path $p$. Similarly, we let $\mathcal A _2 (p)$ be the event that the first $|p|$ steps of $X_2$ were along the path $p$. Note that the events $\mathcal A _1 (p)$ and $\mathcal A _2 (p)$ are independent of each other and independent of the Poisson clocks on the vertices. We have that
\begin{equation}\label{eq:7}
    \mathbb P (\mathcal A _1 (p) ) = \mathbb P (\mathcal A _2 (p) ) = \prod _{j=0}^{|p|-1} P\big( p(j),p(j+1) \big).
\end{equation}

Finally, let $N_1(t)$ and $N_2(t)$ be the number of steps taken by the Markov chains $X_1$ and $X_2$ respectively up to time $t$. Clearly $N_1(t),N_2(t) \sim \text{Poisson}(t)$. It is convenient to work inside the event $\mathcal B :=\{ |N_1(t)+N_2(t)-2t| \le t \}$ which holds with high probability. Indeed,   
\begin{equation}\label{eq:B^c}
    \mathbb P (\mathcal B ^c) \le \mathbb P ( |N_1(t)-t|\ge t/2 )+\mathbb P ( |N_2(t)-t|\ge t/2 ) \le Ce^{-ct}.
\end{equation}

We have that 
\begin{equation}
    \{ X_1(t)=X_2(t) \}\cap \mathcal B = \bigcup _{ n_1,n_2  } \bigcup _{(q,q')} \mathcal A _1 (q) \cap \mathcal A _2 (q') \cap \{ N_1(t)=n_1 \} \cap \{ N_2(t)=n_2 \}
\end{equation}
where the first union is over pairs of integers $n_1$ and $n_2$ with $|n_1+n_2-2t|\le t$ and the second union is over pairs of paths $q$ and $q'$ with $|q|=n_1$ and $|q'|=n_2$ that start at $o$ and end at the same vertex.
Next, let $C_n$ be the set of paths $p$ from $o$ to $o$ with $|p|=n$ in $V$. Clearly, the paths $q$ and $q'$ appearing in the union above can be concatenated to a path $p\in C_n$ where $n:=n_1+n_2$ such that we have $q=p_{n_1}$ and $q'=(p^{-1})_{n_2}$. Thus, by the union bound, 
\begin{equation}
\begin{split}
    \mathbb P (X_1(t)=X_2&(t),\, \mathcal B) \le \sum _{n=\lceil t\rceil }^{\lfloor 3t \rfloor  } \sum _{p\in C_n} \sum _{n_1=0}^{n} \mathbb P (\mathcal A _1 (p_{n_1}) )\cdot \mathbb P (\mathcal A _2 ((p^{-1})_{n-n_1}) ) \cdot \\
    &\quad \quad \quad \quad \quad  \cdot \mathbb P \big( N_1(t)=n_1,\ N_2(t)=n-n_1 \ | \ \mathcal A _1 (p_{n_1}), \mathcal A _2 ((p^{-1})_{n-n_1}) \big) \\
    \le \Delta   \sum _{n=\lceil t\rceil }^{\lfloor 3t \rfloor  } & \sum _{p\in C_n} \mathbb P (\mathcal A_1 (p)) \sum _{n_1=0}^{n} \mathbb P \big( N_1(t)=n_1,\ N_2(t)=n-n_1 \ | \ \mathcal A_1 (p_{n_1}), \mathcal A_2 ((p^{-1})_{n-n_1}) \big) ,
\end{split}
\end{equation}
where in the last inequality we used \eqref{eq:7} and the fact that the chain is reversible. We claim that for any $p\in C_n$,
\begin{equation}\label{eq:6}
\begin{split}
    \mathbb P \big( N_1(t)=n_1,&\ N_2(t)=n-n_1 \ | \ \mathcal A_1 (p_{n_1}), \mathcal A_2 ((p^{-1})_{n-n_1}) \big) \\
    =&\mathbb P \big(  N_1(t)=n_1,\ N_2(t)=n-n_1 \ | \ \mathcal A_1 (p), \mathcal A_2 (p^{-1}) \big) .
    \end{split}
\end{equation}
Indeed, the event $\{N_1(t)=n_1,\ N_2(t)=n-n_1  \}$ is independent of the trajectory of the chain $X_1$ after $n_1$ steps and the trajectory of $X_2$ after $n-n_1$ steps and therefore we can extend these trajectories arbitrarily. Substituting this we get  
\begin{equation}\label{eq:4}
    \mathbb P (X_1(t)=X_2(t), \, \mathcal B) \le  \Delta \sum _{n=\lceil t \rceil }^{ \lfloor 3t \rfloor } \sum _{p\in C_n} \mathbb P ( \mathcal A _1 (p)) \cdot \mathbb P \big(  N_1(t)+N_2(t)=n \ | \ \mathcal A _1 (p), \mathcal A_2 (p^{-1}) \big).
\end{equation}

We need the following lemma in order to estimate the sum on the right-hand side of \eqref{eq:4}. The statement of the lemma requires additional notations. Let $0=t_0<t_1<t_2<\cdots$ be the times in which either one of the Poisson clocks $N_1(t)$ or $N_2(t)$ ring (i.e., the jump times of $N_1(t)+N_2(t)$). We also let $\mathcal Z_\infty$ denotes the $\sigma$-algebra generated by $(X_1(t_i),X_2(t_i),N_1(t_i),N_2(t_i))_{i=1}^{\infty}$. $\mathcal Z_\infty $ contains all the discrete information of the continuous-time Markov chain $(X_1(t),X_2(t))$. 

\begin{lem}\label{lem:3}
For all $t\le n \le 3t$ we have almost surely
\begin{equation}\label{eq:1}
    \mathbb P \big( N_1(t)+N_2(t)=n \ | \ \mathcal Z_\infty \big) \le \frac{C}{\sqrt{t}} \exp \Big( - \frac{c(n-2t)^2}{t} \Big).
\end{equation}
\end{lem}


In order to prove the lemma we will need the following technical claim. To this end, recall that the convolution of two functions $\varphi ,\psi :\mathbb R _+ \to \mathbb R  $ is given by 
\begin{equation}
    \varphi * \psi (x):=\int _0^{x}   \varphi (y) \psi (x-y) dy.
\end{equation}

\begin{claim}\label{claim:4}
Let $n\ge 2$ and let $k_1,k_2>0$ such that $2k_1+k_2=n$. Let $\varphi _1$, $\varphi _2$ and $\varphi _3$ be the densities of $\text{Gamma}(k_1,1)$, $\text{Gamma} (k_2,2)$ and $\text{Gamma}(\lfloor n/2 \rfloor ,1)$ respectively. Then, for all $x\le n$ we have that $\varphi _1 *\varphi _2 (x)\le C\varphi _3(x)$ for some absolute constant $C>0$.
\end{claim}

\begin{proof}
Let $\varphi _4 $ and $\varphi _5$ be the densities of $\text{Gamma}(k_2/2,1)$ and $\text{Gamma} (n/2,1)$ respectively. First, we claim that for all $x>0$, $\varphi _2(x) \le C_2\varphi _4(x)$ for some $C_2>0$. Indeed,
\begin{equation}
\begin{split}
    \varphi _2(x)= \frac{2^{k_2}x^{k_2-1} e^{-2x}}{\Gamma (k_2)}= \frac{2^{k_2}x^{k_2/2-1} e^{-x}}{\Gamma (k_2)}x^{k_2/2} e^{-x} \le \frac{2^{k_2}x^{k_2/2-1} e^{-x}}{\Gamma (k_2)}(k_2/2)^{k_2/2} e^{-k_2/2} \\
    \le C_1 \frac{2^{k_2/2}e^{k_2/2}}{\sqrt{k_2}\cdot k_2^{k_2/2}} x^{k_2/2-1}e^{-x}  \le C_2  \frac{x^{k_2/2-1}e^{-x}}{\Gamma (k_2/2)}=C_2\varphi _4(x),
\end{split}
\end{equation}
where in the first inequality we used that $f(x):=x^ae^{-x}$ is maximal when $x=a$ and in the last two inequalities we used Stirling's formula. Thus, using that $k_1+k_2/2=n/2$ we obtain 
\begin{equation}\label{eq:8}
     \varphi _1 *\varphi _2(x) \le C_2\varphi _1 *\varphi _4(x) = C_2 \varphi _5(x) .
\end{equation}
Finally, we have 
\begin{equation}\label{eq:9}
    \varphi _5(x) = \frac{x^{n/2-1}e^{-x}}{\Gamma (n/2)}\le C_3 \frac{x^{\lfloor n/2 \rfloor -1}e^{-x}}{\Gamma (\lfloor n/2 \rfloor )} \le C_3 \varphi _3(x),
\end{equation}
where the first inequality holds trivially when $n$ is even and when $n$ is odd it follows from Stirling's formula using that $x\le n$. The claim follows from \eqref{eq:8} and \eqref{eq:9} (with $C=C_2C_3$).
\end{proof}

We now turn to prove Lemma~\ref{lem:3}.

\begin{proof}[Proof of Lemma~\ref{lem:3}]
Condition on $\mathcal Z _{\infty }$ and let 
\begin{equation}
    k:=\min \big\{ i\ge 1 : N_1(t_i)+N_2(t_i)\ge n \big\}.
\end{equation}
Without loss of generality, we may assume that $N_1(t_k)+N_2(t_k)=n$. Indeed, otherwise the conditional probability in the left-hand side of  \eqref{eq:1} is $0$ and the lemma follows. Suppose in addition that $X_1(t_k)=X_2(t_k)$. The case in which $X_1(t_k)\neq X_2(t_k)$ is similar.

Next, note that if $X_1(t_i)=X_2(t_i)$ then $t_{i+1}-t_i\sim \exp (1)$ and if $X_1(t_i)\neq X_2(t_2)$ then $t_{i+1}-t_i\sim \exp (2)$. Thus,
\begin{equation}\label{eq:2}
    (t_k,t_{k+1}) \overset{d}{=} \Big( \sum _{k=1}^{k_1} Z_k+ \sum _{k=1}^{k_2} W_k , \sum _{k=1}^{k_1+1} Z_k+ \sum _{k=1}^{k_2} W_k \Big),
\end{equation}
where $Z_1,Z_2,\dots $ are i.i.d.\ random variables with $\exp (1)$ distribution, $W_1,W_2,\dots $ are i.i.d.\ random variables with $\exp (2)$ distribution that are independent of $Z_1,Z_2,\dots $ and where 
\begin{equation}
    k_1:=|\{ i< k : X_1(t_i)=X_2(t_i) \}|,\quad  k_2:=|\{ i< k : X_1(t_i)\neq X_2(t_i) \}|.
\end{equation}
Note that $k,k_1,k_2$ are measurable in $\mathcal Z _{\infty }$ while $Z_i$ and $W_i$ are independent of $\mathcal Z _{\infty }$. Moreover, note that $2k_1+k_2=n$ since $N_1(t_i)+N_2(t_i)$ jumps by $2$ when $X_1(t_i)=X_2(t_i)$ and jumps by $1$ otherwise.

Conditioning on the value of $Z_{k_1+1}$ we obtain that 
\begin{equation}\label{eq:poisson bound}
\begin{split}
    \mathbb P \big( N_1(t)+N_2(t)=n \ | \ \mathcal Z _{\infty } \big)  &= \mathbb P \big( t_k\le t \le t_{k+1} \ | \ \mathcal Z _{\infty } \big)  \\
   &=\intop _0^{\infty } \mathbb P  \big( t-z\le t_k\le t \ | \ \mathcal Z _{\infty }\big)    e^{-z}dz.
\end{split}
\end{equation}
Finally, by \eqref{eq:2}, the density of $t_k$ conditioning on $\mathcal Z _{\infty }$ is exactly the density $\varphi _1*\varphi _2$ from Claim~\ref{claim:4}. Thus, if we let $T$ and $Z$ be independent random variables with $T\sim \text{Gamma}(\lfloor n/2 \rfloor ,1)$ and $Z\sim \exp (1)$, then by Claim~\ref{claim:4} and using that $n\ge t$ we have 
\begin{equation}\label{eq:10}
\begin{split}
    \intop _0^{\infty } \mathbb P &\big(  t-z\le t_k\le t  \ | \ \mathcal Z _{\infty } \big)   e^{-z}dz \le C \intop _0^{\infty } \mathbb P ( t-z\le T\le t )   e^{-z}dz \\
    &= C \cdot \mathbb P ( T\le t \le T+Z)=C \cdot \mathbb P (N=\lfloor n/2 \rfloor ) \le \frac{C}{\sqrt{t}} \exp \Big( -\frac{c(n-2t)^2}{t} \Big),
\end{split}
\end{equation}
where $N$ is a Poisson$(t)$ random variable and where the last inequality is a standard estimate on the Poisson distribution. Substituting \eqref{eq:10} into \eqref{eq:poisson bound} finishes the proof of \eqref{eq:1}.
\end{proof}

We now proceed with the proof of Proposition~\ref{prop:fragmentation infinite-graph second moment}. The events $\mathcal A _1(p)$ and $\mathcal A _2(p)$ are measurable in $\mathcal Z _{\infty }$. Thus, substituting the bound \eqref{eq:1} into \eqref{eq:4} we obtain  
\begin{equation}\label{eq:5}
\begin{split}
    \mathbb P &(X_1(t)=X_2(t),\, \mathcal B)\le  \frac{C\Delta }{\sqrt{t}} \sum _{n=\lceil t \rceil }^{ \lfloor 3t \rfloor } \exp \Big( -\frac{c(n-2t)^2}{t} \Big) \sum _{p\in C_n}  \mathbb P (\mathcal A (p))\\
    &=\frac{C\Delta }{\sqrt{t}} \sum _{n=\lceil t \rceil }^{ \lfloor 3t \rfloor } \exp \Big(- \frac{c(n-2t)^2}{t} \Big) \cdot P_{o,o}^n \le \frac{C\Delta C_0}{t} \sum _{n=\lceil t \rceil }^{ \lfloor 3t \rfloor } \exp \Big( - \frac{c(n-2t)^2}{t} \Big)\le \frac{\tilde{C}\Delta C_0}{\sqrt{t}},
\end{split}
\end{equation}
where in the second inequality we used Assumption~\ref{assumption 1/sqrt t} and the fact that $t\le n\le 3t\le |V|^2$. This finishes the proof of the proposition using \eqref{eq:B^c}.

\subsection{Proof of Proposition~\ref{prop:fragmentation infinite-graph second moment}}

We first provide an informal overview of the proof. The event $\{X_1(t)=X_2(t)\}$ is partitioned into $ O(t\log t)$ events and Assumption~\ref{assumption 1/sqrt t} is used for estimating the probability of these events as follows: 
\begin{multline}\label{eq:big picture}
\mathbb P(X_1(t)=X_2(t))
\approx \sum_{n_1,n_2 = t\pm O(\sqrt{t\log t})}\mathbb P(\hat X_1(n_1)=\hat X_2(n_2),N_1(t)=n_1,N_2(t)=n_2)\\
\lesssim \frac {C_0}{\sqrt t} \sum_{n_1,n_2}
\mathbb P(N_1(t)=n_1,N_2(t)=n_2|\hat X_1(n_1)=\hat X_2(n_2)).
\end{multline}
To estimate the RHS of \eqref{eq:big picture}, we analyze the $\mathbb Z_+^2$-valued process $\{N(s)\}_s=\{(N_1(s),N_2(s))\}_s$ conditioned on $\hat X_1,\hat X_2$, where $\hat X_i$ is the trajectory of $X_i$, namely, $X_i(t)=\hat X_i(N_i(t))$. The process $N(\cdot)$ is a continuous-time simple random walk on a directed graph whose vertex set is $\mathbb Z_+^2$. There are unit Poisson clocks on the edges. The initial state is $(0,0)$, and the outgoing edges from each state $(z_1,z_2)\in\mathbb Z^2$ depend on $\hat X_1,\hat X_2$ as follows: 
\[
(z_1,z_2)\to
\begin{cases}
(z_1+1,z_2+1)&\text{if $\hat X_1(z_1)=\hat X_2(z_2)$,}\\
(z_1+1,z_2),(z_1,z_2+1)&\text{if $X_1(z_1)\neq\hat X_2(z_2)$.} 
\end{cases}
\]

For $n\in\mathbb N$, consider the time in which the random walk crosses the line $z_1+z_2=n$ denoted by  
\[
\tau(n):=\min\{s>0:N_1(s)+N_2(s)\geq n\}.
\]
The event $\{N_1(t)=n_1,N_2(t)=n_2\}$ from \eqref{eq:big picture} is expressed as $E_1\cap E_2$ where 
\begin{align}
&E_1:= \{N_1(\tau(n_1+n_2))-N_2(\tau(n_1+n_2))=n_1-n_2\},\\
&E_2:=\{N_1(t)+N_2(t)=n_1+n_2\}.
\end{align}
The event $E_1$ depends only on the trajectory of $N(\cdot)$. The probability of $E_2$ conditioned on the trajectory of $N(\cdot)$ is of order $O(t^{-\frac 1 2})$, by Lemma~\ref{lem:3}. Therefore, it remains to estimate the probability of $E_1$ conditioned on $\hat X_1,\hat X_2$. Specifically, it remains to show that for every $t$ large enough (depending on $C_0$), every $x\approx t $ and every $y\in\mathbb Z$, 
\begin{equation}\label{eq: informal t^{-a}}
\mathbb P(N_1(\tau(x))-N_2(\tau(x))=y|\hat X_1,\hat X_2 )\lesssim t^{-\alpha},
\end{equation}
 for some universal constant $\alpha>0$.  

Most of the effort in the proof of Proposition~\ref{prop:fragmentation infinite-graph second moment} is devoted to proving \eqref{eq: informal t^{-a}}. Note, first, that \eqref{eq: informal t^{-a}} does not hold at every realization of $\hat X_1,\hat X_2$. For example, if $\hat X_1=\hat X_2$, then $N_1(s)=N_2(s)$ for every $s$. We identify a class of realizations of $\hat X_1,\hat X_2$, for which \eqref{eq: informal t^{-a}} does hold and call them ``good trajectories.'' The notion of good trajectories is not binary, it is rather a spectrum.  The goodness of the trajectories determines the value of $\alpha$: the better the trajectories the larger an $\alpha$ can be guaranteed. It turns out that to show that the trajectories are good with high probability, it is sufficient to bound from below the expected proportion of time $c$ in which $X_1(s)\neq X_2(s)$, and the greater the lower bound the greater the ``goodness'' guarantee. In our proof, we first obtain a naive lower bound on $c$ depending on $C_0$ and use it to obtain an $\alpha$ that depends on $C_0$. Then, from \eqref{eq: informal t^{-a}}, we immediately get $c$ arbitrarily close to $1$, and with it get a universal $\alpha$.      

We turn now to the formal proof of  Proposition~\ref{prop:fragmentation infinite-graph second moment}. Let us introduce some notations and conventions first. 
\begin{itemize}
    \item We work under Assumption~\ref{assumption 1/sqrt t} with parameter $C_0$.    
    \item A pair of vertices $v,u\in V$ induce a probability measure over $X_1(t),X_2(t)$ by setting $X_1(0)=v, X_2(0)=u$. This probability measure is denoted $\mathbb P_{v,u}$ and the respective expectation is denoted $\mathbb E_{v,u}$.
    \item The (independent) trajectories of $X_1(t)$ and $X_2(t)$ ($t\geq 0$) are denoted $\hat X_1(n)$ and $\hat X_2(n)$ ($n=0,1,2,\ldots$) respectively.
    \item The (unit) Poisson clock that counts the steps of each $X_i$ is denoted $N_i(t)$. That is, $X_{i}(t)=\hat X_i(N_i(t))$ ($i=1,2$, $t\geq 0$).
    \item The times in which either one of the chains progress are denoted $(t_n)_{n=1}^\infty$, i.e., these are the step times of $N_1(t)+N_2(t)$. 
    \item The joint trajectory is defined as $Z(n):=(X_1(t_n),X_2(t_n),N_1(t_n),N_2(t_n))$, and 
$Z_i(n):=X_i(t_n)$, $i=1,2$.
    \item The $\sigma$-algebra generated by $Z(0),\ldots,Z(n)$ is denoted $\mathcal Z_n$ and the $\sigma$-algebra generated by $\{Z(n)\}_{n=0}^\infty$ is denoted $\mathcal Z_\infty$. \item The set $\{1,\ldots,m\}$ is denoted $[m]$, and $n+[m]$ denotes the set $\{n+1,\ldots,n+m\}$.
\end{itemize}   
\begin{lem}\label{lemma delta m}
Suppose there exist $\delta\in (0,1)$ and $m\in \mathbb N$ such that $P_{vu}^m\leq 1-\delta$, for every $v,u\in V$. Then, 
\[
\mathbb E_{v,u} \big[ \#\{1\leq i\leq m : Z_1(i)\neq Z_2(i) \} \big] \geq \delta,
\]
for every $v,u\in V$.
\end{lem}
\begin{proof}
Denote $Y=\#\{1\leq i\leq m : Z_1(i)\neq Z_2(i) \}$. It suffices to show $\mathbb P(Y=0)\leq 1-\delta$. Assume first that $v\neq u$. Since  
\[\{Y=0,N_1(t_1)=1\}\subset \{\hat X_1(i+1)=\hat X_2(i):i=0,\ldots,m-1\}\subset\{\hat X_1(m)=\hat X_2(m-1)\},
\] 
we have
\begin{multline}
    \mathbb P(Y=0,N_1(t_1)=1)\leq \mathbb P(\hat X_1(m)=\hat X_2(m-1),N_1(t_1)=1)\\
    =\mathbb P(\hat X_1(m)=\hat X_2(m-1))\mathbb P(N_1(t_1)=1)\leq (1-\delta)\frac 1 2.
\end{multline}
Similarly,
\begin{multline}
    \mathbb P(Y=0,N_1(t_1)=0)\leq \mathbb P(\hat X_1(m-1)=\hat X_2(m),N_1(t_1)=0)\\
    =\mathbb P(\hat X_1(m-1)=\hat X_2(m))\mathbb P(N_1(t_1)=0)\leq (1-\delta)\frac 1 2.
\end{multline}
Adding together the two inequalities yields $\mathbb P(Y=0)\leq 1-\delta$, which completes the proof for the case $v\neq u$. In the case $v=u$, we the proof is concluded with 
\begin{equation}
    \mathbb P(Y=0)= \mathbb P(\hat X_1(1)=\hat X_2(1),\ldots,\hat X_1(m)=\hat X_2(m))
    \leq \mathbb P(\hat X_1(m)=\hat X_2(m))\leq 1-\delta.
\end{equation}

\end{proof}
\begin{lem}\label{lem: good k steps}
Suppose there exists $m\in\mathbb N$ and $c>0$ such that
\[
\mathbb E_{v,u}\big[ \#\{1\leq i\leq m : Z_1(i)\neq Z_2(i) \} \big] \geq cm,
\]
for every $v,u\in V$.
Then, for every $k\geq m$, every $\eps>0$ and every $v,u\in V$, 
\[
\mathbb P_{v,u} \Big( \mathbb E_{v,u} \big[ \#\{1\leq i\leq k:Z_1(i) \neq Z_2(i)\}\mid \hat X_1,\hat X_2\big] >c_1 k \Big) \geq 1- \exp(-{c_2}k),
\]
where $c_2=\frac {\eps ^2}{4m}(1 - \frac m k)$ and $c_1=(1-\sqrt{1-c})(1-\sqrt{1-c}-\eps)(1-\tfrac m k)(1- \exp(-{c_2}k))$.
\end{lem}
\begin{proof}[Proof of Lemma~\ref{lem: good k steps}]
Fix $k,\eps,v,u$. For a finite set $A\subset\mathbb N$, define a random variable
$Y_A:=\frac 1 {|A|} \#\{i\in A: Z_1(i)\neq Z_2(i) \}$. We must show that \[\mathbb P_{v,u}(\mathbb E_{v,u}[Y_{[k]}\mid \hat X_1,\hat X_2\big] >c_1])\geq 1- \exp(-{c_2}k).\]

By Markov's Inequality,
\[
\mathbb P_{v,u}(1- Y_{n+[m]}\geq \sqrt {1-c}\mid \mathcal Z_n)\leq \sqrt{1-c},\quad \forall n\in\mathbb N.
\]
It follows by Azuma's Inequality that
\[
\mathbb P_{v,u}\left(\frac{1}{\lfloor k/m\rfloor}\sum_{i=1}^{\lfloor k/m\rfloor}\mathbf{1}_{\{1-Y_{(i-1)m+[m]}\geq\sqrt{1-c}\}}> \sqrt{1-c} +\eps \right)\leq \exp\left(-\tfrac {\eps^2}{2}\lfloor k/m\rfloor\right).
\]
Therefore, by Markov's Inequality,
\begin{multline}
\mathbb P_{v,u}\left(\mathbb P_{v,u}\left(\frac{1}{\lfloor k/m\rfloor}\sum_{i=1}^{\lfloor k/m\rfloor}\mathbf{1}_{\{1-Y_{(i-1)m+[m]}\geq\sqrt{1-c}\}}> \sqrt{1-c} +\eps \middle |\hat X_1,\hat X_2\right)
\geq \exp\left(-\tfrac {\eps^2}{4}\lfloor k/m\rfloor\right)\right)\\
\leq \exp\left(-\tfrac {\eps^2}{4}\lfloor k/m\rfloor\right).
\end{multline}
In the event
\[
\left\{
\frac{1}{\lfloor k/m\rfloor}\sum_{i=1}^{\lfloor k/m\rfloor}\mathbf{1}_{\{1-Y_{(i-1)m+[m]}\geq\sqrt{1-c}\}}\leq \sqrt{1-c} +\eps 
\right\},
\]
we have
\[
Y_{[m\lfloor k/m\rfloor]}\geq (1-\sqrt{1-c})(1-\sqrt{1-c}-\eps).
\] 
Since $m\lfloor k/m\rfloor /k \geq (1-\frac m k)$, we have $Y_{[k]}    \geq (1-\frac m k)Y_{[m\lfloor k/m\rfloor]}$. Therefore,
\begin{multline}
\mathbb P_{v,u}\left(\mathbb P_{v,u}\left(Y_{[k]}\geq (1-\sqrt{1-c})(1-\sqrt{1-c}-\eps)(1-\tfrac m k)\middle |\hat X_1,\hat X_2\right)\geq 1- \exp\left(-c_2k\right)\right)\\
\geq 1- \exp\left(-c_2k\right),
\end{multline}
which concludes the proof of Lemma~\ref{lem: good k steps}.
\end{proof}
\begin{definition}[Good trajectories]\label{def: good trajectories}
Let $c,t>0$ and $k\in\mathbb N$. A realization $(p_1,p_2)$ of $\hat X_1,\hat X_2$ is called \emph{$(c,t,k)$-good trajectories} if for every $0 \leq m \leq 3t$,
\[
\mathbb E\left[\#\{m\leq i\leq m+k-1:Z_1(i) \neq Z_2(i)\} \ \big| \ \hat X_1=p_1,\hat X_2=p_2,Z(0),\ldots,Z(m-1)\right]\geq ck.
\]
\end{definition}
\begin{lem}\label{lem: weak good trajectories}
 Let $m\in \mathbb N$ and $c\in(0,1)$ satisfying the condition of Lemma~\ref{lem: good k steps}, and let $0<\eps<1-\sqrt{1-c}$. There exists $t_0>0$ (that depends on $m$ and $\eps$) such that for any $t>t_0$ the following statement holds. Let $k=\lceil\log^2(t)\rceil$  and $c_1$ be as guaranteed by Lemma~\ref{lem: good k steps}. Then,
\[
\mathbb P\left(\text{$(\hat X_1,\hat X_2)$ are $(c_1,t,k)$-good trajectories} \right)
\geq 1 - t^{-5}.
\]
\end{lem}

\begin{proof}[Proof of Lemma~\ref{lem: weak good trajectories}]
Given $m,c$ satisfying the condition of Lemma~\ref{lem: good k steps}, $\eps>0$ and $k\geq m$, let $c_1,c_2>0$ be as guaranteed by Lemma~\ref{lem: good k steps}. 
For a realization $p=(p_1,p_2)$ of $\hat X_1, \hat X_2$ and $k\in\mathbb N$, let
\[
y_k(p):=\mathbb E[\#\{0\leq i \leq k-1: Z_1(i)\neq Z_2(i)\}\mid \hat X_1 =p_1, \hat X_2=p_2].
\]
For a sequence $a=(a_0,a_1,\ldots)$, let $a^{\geq i}$ be the $i$-th suffix of $a$, namely,  $a^{\geq i}:=(a_i,a_{i+1},\ldots)$.

Consider the events
\[
\mathcal A_k^{i,j}:=\{y_k(\hat X_1^{\geq i},\hat X_2^{\geq j})\leq c_1k\}.
\]
Lemma~\ref{lem: good k steps} ensures that $\mathbb P(\mathcal A_k^{i,j})< e^{-c_2k}$, for every $i,j\in \mathbb N$. Outside the event 
$\bigcup\limits_{0\leq i,j\leq 3t}\mathcal A_k^{i,j}$, the trajectories $\hat X_1,\hat X_2$ are $(c_1,t,k)$-good, therefore, by the union bound,
\[
\mathbb P\left(\bigcup\limits_{0\leq i,j\leq 3t}\mathcal A_k^{i,j}\right)\leq 9t^2e^{-c_{2}k}\leq 9t^{2-c_2\log t}\leq t^{-5},
\]
for every $t$ large enough depending on $m$ and $\eps$. This concludes the proof of Lemma~\ref{lem: weak good trajectories}.
\end{proof}
Following is an immediate corollary of \cite[Theorem~1]{gurel2014localization}.
\begin{cor}\label{cor: bounded variance martingale}
For every $c\in(0,1)$ there exists $\alpha=\alpha(c)>0$ such that for every $n\in\mathbb N$ and every martingale $S_1,\ldots, S_n$ the following holds. Suppose, for every $1\leq i \leq n-1$,
\begin{itemize}
    \item $|S_{i+1}-S_i|\leq n$, and
    \item $\mathrm{Var}(S_{i+1}\mid S_1,\ldots,S_i)\in [c,1].$
\end{itemize}
Then, 
\[
\sup_{x}\mathbb P(S_n=x)\leq n^{-\alpha}.
\]
\end{cor}

To proceed we need some further notations.
\begin{itemize}
    \item Let $S(n):=N_1(t_n)-N_2(t_n)$. Note that $S(n)$ is a martingale adapted to $\mathcal Z_n$ conditioned on $\hat X_1$ and $\hat X_2$.
    \item Let $\tau(n):=\min\{m:N_1(t_m)+N_2(t_m)\geq n\}$. $\tau (n)$ is a stopping time adapted to $\mathcal Z _n$.
\end{itemize}
\begin{lem}\label{lem: good trajectories}
For every $c>0$, there exists $\alpha >0$ such that for every $t$ large enough (depending on $c$) the following holds. Let $k:=\lceil(\log t)^2\rceil$,  $t < n <  3t$, and let $(p_1,p_2)$ be  $(c,t,k)$-good trajectories. Then,
\[
\forall x\quad \mathbb P(S(\tau(n))=x\mid \hat X_1=p_1,\hat X_2=p_2)\leq t^{-\alpha}.
\]
Furthermore, $\alpha$ can be arbitrarily close to $\alpha(\max\{2c-1,c/\sqrt 2\})$ of Corollary~\ref{cor: bounded variance martingale}.
\end{lem}
\begin{proof}[Proof of Lemma~\ref{lem: good trajectories}]
We condition on the event
\begin{equation}
 \big\{ \forall m\leq n,\  \hat X_1(m)=p_1(m) \text{ and } \hat X_2(m)=p_2(m) \big\} .
\end{equation}
We first assume that $c>1/2$ and prove the lemma with $\alpha$ arbitrarily close to $\alpha(2c-1)$ of Corollary~\ref{cor: bounded variance martingale}. For simplicity, we assume that $n$ is divisible by $k$.
Consider the martingale $S_0,S_1,\ldots,S_{n/k}$ defined by
\[
S_i:=S(\tau(ki)).
\]
[If $n$ is not divisible by $k$, then since $n>k^2$, one can find integers $k_1,k_2,\ldots,k_{\lfloor n/k\rfloor}$, such that $k_i\in \{k,k+1\}$ and $\sum_i k_i=n$. Let $K_i:= \sum_{\ell=1}^i k_\ell$. The definition of $S_i$ is then amended as $S_i=S(\tau(K_i))$.] 

We would like to apply Corollary~\ref{cor: bounded variance martingale}. Let us verify its conditions. The increments $|S_{i+1}-S_i|$ are bounded by $k$, since 
\begin{equation}\label{eq:S increaments}
    |S(\ell+1)-S(\ell)|=
    \begin{cases}
    1 & \text{iff $(N_1(t_{\ell+1})+N_2(t_{\ell+1}))-(N_1(t_{\ell})+N_2(t_{\ell}))=1$,}\\
    0 & \text{iff $(N_1(t_{\ell+1})+N_2(t_{\ell+1}))-(N_1(t_{\ell})+N_2(t_{\ell}))=2$,}
    \end{cases}
\end{equation}
for every $\ell\in\mathbb N$. 

We now need to bound the variance of the increments of $S_i$. Fix $i\in [n/k]$. The quadratic variation of $S(n)$ between time $\tau(ki)$ and $\tau(ki+\ell)$ is denoted
\begin{equation}
\label{eq:quadratic variation}
\sigma_{\ell}:=\sum_{ m=\tau(ki) }^{ \tau(ki)+\ell -1} \big( S(m+1)-S(m) \big) ^2,
\end{equation}
for every $\ell\in\mathbb N$. 

Let $\tau:=\tau(k(i+1))-\tau (ki)$. We have
\begin{equation}\label{eq:variance S_i}
\mathrm{Var}(S_{i+1}\mid S_0,\ldots,S_i)=\mathbb E[\sigma_{\tau}\mid  \mathcal Z_{\tau(ki)}].
\end{equation}
Note that $\sigma_\ell$ counts the number of $1$-increments up to time $\ell$ and $\ell-\sigma_\ell$ the number of $0$-increments. By \eqref{eq:S increaments} and \eqref{eq:quadratic variation},
\[
\tau+(\tau-\sigma_\tau)\in\{k,k+1\}.
\]
Since $\tau\leq k$, we have $\sigma_\tau\leq \sigma_k$ and $\tau-\sigma_\tau\leq k-\sigma_k$. It follows that $\sigma_\tau\leq \sigma_k\leq \tau\leq k$. Thus, we get
\[
2\sigma_k-k= k-2(k-\sigma_k)\leq \sigma_\tau\leq k+1.
\]
Since $(p_1,p_2)$ are $(c,t,k)$-good trajectories, we have 
$\mathbb E[\sigma_k|\mathcal Z_{\tau(ki)}]\geq ck$. By \eqref{eq:variance S_i}, we get
\[
(2c-1)k\leq \mathbb \mathrm{Var}(S_{i+1}\mid S_0,\ldots,S_i)\leq k+1.
\]
By Corollary~\ref{cor: bounded variance martingale} applied to $\frac {S_i}{\sqrt{k+1}}$, there exists $\alpha=\alpha(2c-1)>0$ such that
\[
\forall x\quad \mathbb P(S(\tau(n))=x\mid \hat X_1=p_1,\hat X_2=p_2)\leq \lfloor n/k\rfloor ^{-\alpha}< t^{-\alpha+\beta},
\]
for arbitrarily small $\beta>0$ and all $t$ large enough.

To complete the proof of the lemma, we relax the assumption that $c>1/2$ and explain how to amend the proof so as to obtain $\alpha$ arbitrarily close to $\alpha(c/\sqrt 2)$. Amend the definition of $S_i$ to be $S_i:=S(\tau(2ki))$. As before, it can be assumed that $n$ is divisible by $2k$. Fix $i$. Amend the definition of $\tau$ to be $\tau(2k(i+1))-\tau (2ki)$. As before, we get $k\leq \tau\leq 2k+1$, and therefore, $\sigma_k\leq \sigma_\tau\leq 2k+1$. If follows that 
\[
ck\leq \mathrm{Var}(S_{i+1}\mid S_0,\ldots,S_i)\leq 2k+1.
\]
Consider the normalized martingale $S_i/\sqrt{2k+1}$ and apply Corollary~\ref{cor: bounded variance martingale} to get the desired bound.
\end{proof}

The following lemma is a weaker version of Proposition~\ref{prop:fragmentation infinite-graph second moment}. Here, the number $\alpha$ depends on $C_0$ rather than being a universal constant.
\begin{lem}\label{lem: weak prop second moment}
There exist numbers $t_0,\alpha >0$ depending on $C_0$, where $\alpha$ depends on $C_0$ only through the number $c_1$ guaranteed by Lemma~\ref{lem: weak good trajectories}, such that for any $t\in (t_0,|V|^2/2)$
\[
\mathbb P(X_1(t)=X_2(t))\leq t^{-\alpha}.
\]
Furthermore, the exponent $\alpha$ can be arbitrarily close to the number $\alpha(c_1)$ guaranteed by Corollary~\ref{cor: bounded variance martingale}.
\end{lem}
\begin{proof}[Proof of Lemma~\ref{lem: weak prop second moment}]
Assumption~\ref{assumption 1/sqrt t} implies the conditions of Lemma~\ref{lemma delta m}, say, with $\delta=\frac 1 2$ and $m=\lceil(2C_0)^2\rceil$). Then, $m$ and $c=\delta/m$ and $\eps\in (0,1-\sqrt{1-c})$ satisfy the condition of Lemma~\ref{lem: weak good trajectories}.

Let $t>t_0>0$, $k:=\lceil (\log t)^2 \rceil$, and $c_1>0$ be given by Lemma~\ref{lem: weak good trajectories}, and $\mathcal A_t$ the event that $(\hat X_1,\hat X_2)$ are $(c_1,t,k)$-good trajectories. Let $I_t:=\{n\in\mathbb N:|n-t|\leq C\sqrt{t\log t}\}$, where $C>0$ is a universal constant such that $\mathbb P(|N_1(t)-t|>C\sqrt{t\log t})<t^{-5}/2$. We have
\begin{equation}\label{eq: X_1(t)=X_2(t)}
\begin{split}
    \mathbb P\big(&X_1(t)=X_2(t) \big) \leq 
    \mathbb P \big( N_1(t)\not\in I_t \big)  +\big(N_2(t)\not\in I_t \big)  + \mathbb P (\mathcal A_t^c) +\\
    &+\sum_{ n_1,n_2 \in I_t}\mathbb P \big( \hat X_1(n_1)=\hat X_2(n_2) \big) \mathbb P \big( N_1(t)=n_1,N_2(t)=n_2\mid \hat X_1(n_1)=\hat X_2(n_2),\mathcal A_t \big) \\
    &\leq 2t^{-5} + 2C_0 t^{-1/2}\sum_{n_1,n_2\in I_t}\mathbb P(N_1(t)=n_1,N_2(t)=n_2\mid \hat X_1(n_1)=\hat X_2(n_2),\mathcal A_t).
\end{split}
\end{equation}
 Where, the last inequality follows from Assumption~\ref{assumption 1/sqrt t}.

By Lemma~\ref{lem:3}, there exists a universal constant $C$ such that for every $t>0$
\begin{equation}\label{eq:N_1+N_2}
\forall x\in I_t+I_t \quad \mathbb P(N_1(t)+N_2(t)=x\mid \mathcal Z_\infty)\leq \frac {C}{\sqrt {t}}.
\end{equation}

Given $n_1,n_2\in I_t$,
\begin{equation}\label{eq: (n_1 - n_2) and (n_1 + n_2)}
\begin{split}
    \mathbb P&(N_1(t)=n_1,N_2(t)=n_2\mid \hat X_1,\hat X_2,\mathcal A_t)\\
    &= \mathbb P(S(\tau(n_1+n_2))=n_1-n_2,\ N_1(t)+N_2(t)=n_1+n_2\mid \hat X_1,\hat X_2,\mathcal A_t)\\
    &=\mathbb P(S(\tau(n_1+n_2))=n_1-n_2\mid \hat X_1,\hat X_2,\mathcal A_t)\cdot \\
    & \quad \quad \cdot \mathbb P(N_1(t)+N_2(t)=n_1+n_2\mid \hat X_1,\hat X_2,\mathcal A_t,S(\tau(n_1+n_2))=n_1-n_2)\\
    &\quad \quad \quad \quad\quad \quad\quad \quad\quad \quad \quad\quad \quad\quad \quad\quad \quad \quad\quad \quad\quad \quad\quad \quad \leq t^{-\alpha}C t^{-1/2}.
\end{split}
\end{equation}
Where, $\alpha=\alpha(c_1)>0$ is given by Lemma~\ref{lem: good trajectories}, and $C$ is the universal constant given by \eqref{eq:N_1+N_2}.

The proof of Lemma~\ref{lem: weak prop second moment} is concluded by plugging \eqref{eq: (n_1 - n_2) and (n_1 + n_2)} into \eqref{eq: X_1(t)=X_2(t)} for all the values of $n_1, n_2$ in the summation.  
\end{proof}
We can now show that the condition of Lemma~\ref{lem: good k steps} holds for every $c\in (0,1)$.
\begin{lem}\label{lem: any c}
For every $c\in (0,1)$, there exists $m$ (depending on $C_0$ and $c$) such that if $m\leq |V|^2/4$ then for  every $v,u\in V$,

\[
\mathbb E_{v,u}[\#\{i\in [m]: Z_1(i)\neq Z_2(i)\}]\geq cm.
\]
\end{lem}
\begin{proof}[Proof of Lemma~\ref{lem: any c}]
Fix $v,u\in V$ and $c\in (0,1)$. Let $A:=\{t\geq 0: X_1(t)=X_2(t)\}$. By Lemma~\ref{lem: weak prop second moment}, there exists $\alpha>0$ and $t_0$ (depending on $C_0$) such that
\[
\mathbb E_{v,u}[|A\cap [0,t]|]\leq t_0+ \int_{t_0}^t s^{-\alpha}\diff s \leq t_0 + 2t^{1-\alpha},
\]
for every $t<|V|^2/2$.

Let $m\leq |V|^2/4$. Let $Y_m:=\#\{i\in [m]: Z_1(i)= Z_2(i)\}$. Note that $|A\cap [t_1,t_{m+1}]|$ is a sum of $Y_m$ unit exponential random variables. Therefore,
\begin{equation}
\mathbb E_{v,u}[Y_m]\leq \mathbb E_{v,u}[|A\cap [0,t_{m+1}]|]\leq \mathbb E_{v,u}[|A\cap [0,2m]|]+\mathbb E_{v,u}[(t_{m+1}-2m)^+].
\end{equation}
Let $S_{m+1}$ be the sum of $m+1$ independent unit exponential random variables. Since $S_{m+1}$ stochastically dominates $t_{m+1}$ (indeed, $S_{m+1}$ can be realized as the time of the $m+1$-th jump of $X_1(t)$),
\[
\mathbb E_{v,u}[Y_m]\leq t_0+4m^{1-\alpha} + \mathbb E[(S_{m+1}-2m)^+]\leq  t_0+4m^{1-\alpha} + me^{-\frac m 8}\leq (1-c)m,
\]
for any $m$ large enough depending on $C_0$ and $c$.
\end{proof}

Proposition~\ref{prop:fragmentation infinite-graph second moment} is concluded from Claim~\ref{claim:1} and the following lemma.
\begin{lem}\label{lem: strong prop second moment}
There exists a universal constant $\alpha >0$ and number $t_0>0$ depending on $C_0$, such that for any $t\in(t_0,|V|^2/2)$,
\[
\mathbb P(X_1(t)=X_2(t))\leq t^{-\alpha}.
\]
\end{lem}
\begin{proof}[Proof of Lemma~\ref{lem: strong prop second moment}]
Lemma~\ref{lem: any c} ensures that there exists $m$ depending on $C_0$ such that 
\[
\mathbb E_{v,u}[\#\{i\in [m]:Z_1(i)\neq Z_2(i)\}]\geq m/2,
\]
as long as $|V|^2\geq 4m$.
Plugging this to Lemma~\ref{lem: good k steps}, provides constants $c_1,c_2>0$ such that
\[
\mathbb P_{v,u}(\mathbb E_{v,u} [\#\{1\leq i\leq k:Z_1(i) \neq Z_2(i)\}\mid \hat X_1,\hat X_2]>c_1 k)\geq 1- \exp(-{c_2}k),
\]
for every $k$ large enough depending on $C_0$ (e.g., $k\geq m^2$).

The proof of Lemma~\ref{lem: strong prop second moment} is concluded by applying Lemma~\ref{lem: weak prop second moment} with the constant $c_1$.
\end{proof}
\begin{remark}\label{remark: imporving alpha}
In the proof of Lemma~\ref{lem: strong prop second moment}, we could have chosen $c_1$ arbitrarily close to $1$. Therefore, Lemma~\ref{lem: strong prop second moment} and Proposition~\ref{prop:fragmentation infinite-graph second moment} holds with respect to any $\alpha<\alpha^*:=\sup\limits_{c\in(0,1)}\alpha(c)$, where $\alpha(c)$ is given by Corollary~\ref{cor: bounded variance martingale}. The result of \cite{gurel2014localization} provides some $\alpha^*<\frac 1 2$. 

We conjecture that the true value of $\alpha^*$ is $\frac 1 2$. An indication to that is a result in \cite{armstrong2016local} (see Corollary~1.1 and equation~(1.10) therein), that shows that in Corollary~\ref{cor: bounded variance martingale}, under the stricter assumption that $|S_i-S_{i-1}|\leq r$,   
\[
\sup_{x} \mathbb P(S_n=x)\leq C(r)n^{-\alpha(c)},
\]
where $\lim\limits_{c\nearrow 1}\alpha(c)=\frac 1 2$. 

It is not shown in \cite{armstrong2016local} how $C(r)$ depends on $r$. We believe that proving that $C(r)$ is polynomial in $r$ should be sufficient for concluding that any $\alpha\in \left(0,\frac 1 2\right)$  can by used in Proposition~\ref{prop:fragmentation infinite-graph second moment}.  
\end{remark}

\subsection{Proof of Proposition~\ref{prop:fragmentation infinite-graph d-th moment}}

Throughout the proof we think of $C_0$, $C_1$ and $\alpha $ as fixed and allow the constants $C$ and $c$ to depend on them.
Let $X_1(t),\ldots,X_d(t)$ be the Markov chains of Claim~\ref{claim:1}. By the assumption of the proposition and Claim~\ref{claim:1} we have for all $i<j\le d$ that 
\begin{equation}
    \mathbb P \big( X_i(t)=X_j(t) \big)\le C_1 t^{-\alpha }.
\end{equation}
For $i<j\le d$ define the set of times 
\begin{equation}
    A_{i,j}:=\{t>0 :  X_{i}(t)=X_{j}(t)\}.
\end{equation}

We start with the following lemma.

\begin{lem}\label{lem:4}
For all $k$ sufficiently large (depending on $C_0,C_1,\alpha $) we have 
\begin{equation}
    \mathbb P \big( |A_{i,j}\cap [0,t]|\ge kt^{1-\alpha } \big) \le e^{-ck},
\end{equation}
where $|A_{i,j}\cap [0,t]|$ denotes the Lebesgue measure of the set $A_{i,j}\cap [0,t]$.
\end{lem}

\begin{proof}
Let $i\neq j$ and $k\in \mathbb N$. We start by bounding the $k$-th moment of $|A_{i,j}\cap [0,t]|$. We have that 
\begin{equation}
\begin{split}
    \mathbb E \big[ |A_{i,j}\cap [0,t]| ^k\big]&=\mathbb E \big[\text{Vol} \big( A_{i,j}^k\cap [0,t]^k \big) \big]= \int _{[0,t]^k} \mathbb P \big( (s_1,\dots ,s_k) \in A_{i,j}^k \big) ds\\
    &=k!\int _0^t \int _{s_1}^t\cdots \int _{s_{k-1}} ^t \mathbb P \big( \forall m\le k, \ s_m \in A_{i,j} \big)  ds_k\cdots ds_1,
\end{split}
\end{equation}
where the last equality is by symmetry. Now, by Proposition~\ref{prop:fragmentation infinite-graph second moment}, for all $0=s_0\le s_1\le \cdots \le s_k$ we have that \begin{equation}
\begin{split}
    \mathbb P \big( \forall m\le k, \ s_m \in A_{i,j} \big)&=\mathbb P \big( \forall m\le k, \ X_i(s_m)=X_j(s_m) \big)\\
    &= \prod _{m=1}^k \mathbb P \left( X_i(s_m)=X_j(s_m) \ \big| \ \forall l<m, \  X_i(s_l)=X_j(s_l)   \right)\\
    &\le \prod _{m=1}^k C_1(s_m-s_{m-1})^{-\alpha }. 
\end{split}
\end{equation}
Thus, we obtain that 
\begin{equation}
    \mathbb E \big[ |A_{i,j}\cap [0,t]| ^k\big]\le k!C^k\int _0^t s_1^{-\alpha } \int _{s_1}^t (s_2-s_1)^{-\alpha }\cdots \int _{s_{k-1}} ^t  (s_k-s_{k-1})^{-\alpha }  ds_k\cdots ds_1 \le C^kk!t^{(1-\alpha )k}.
\end{equation}
By Markov's inequality, we have for all $k\ge 1$
\begin{equation}
    \mathbb P \big( |A_{i,j}\cap [0,t]| \ge C_2 kt^{1-\alpha } \big) \le  C_2^{-k}k^{-k}t^{-(1-\alpha )k} \cdot \mathbb E \big[ |A_{i,j}\cap [0,t]| ^k\big] \le (C/C_2)^k \le e^{-k},
\end{equation}
where the last inequality holds as long as $C_2$ is sufficiently large. This completes the proof of the lemma.
\end{proof}

Let $\tilde{X}_i$ be the trajectory of $X_i$. Namely, $\{\tilde{X}_i(n)\}_{n=0}^\infty$ is the discrete walk such that $X_i(t)=\tilde{X}_i(N_i(t))$. Let $A:=\bigcup _{i\neq j}A_{i,j}$ and note that if $|A\cap [0,t]| \ge t^{1-\alpha /2}$ then there are $i<j\le d$ such that $|A_{i,j}\cap [0,t]| \ge d^{-2}t^{1-\alpha /2} \ge t^{1-4\alpha /5}$. Thus, as long as $t$ is sufficiently large (depending on $C_0,C_1,\alpha $) by Lemma~\ref{lem:4} and a union bound, we have
\begin{equation}
\begin{split}
  \mathbb E \big[  \mathbb P \big( |A\cap [0,t]| \ge t^{ 1-\alpha /2 } \ \big| \ \tilde{X}_i, i\le d &     \big)\big]=\mathbb P \big( |A\cap [0,t]| \ge t^{ 1-\alpha /2 }  \big) \\
  \le \sum _{i< j} \mathbb P \big( |A_{i,j}\cap &[0,t]| \ge t^{ 1-4\alpha /5  }  \big) \le  d^2 \exp 
  (-ct^{\alpha /5}) \le t^{-2d} ,
 \end{split}
\end{equation}
where in the last inequality we used that $d\le t^{\alpha /10}$. Thus, by Markov's inequality
\begin{equation}
    \mathbb P \left(  \mathbb P \left( |A\cap [0,t]| \ge t^{1-\alpha /2 } \ \big| \ \tilde{X}_i, i\le d \right) \ge t^{-d} \right) \le t^{-d}.
\end{equation}
Let $\mathcal P _t$ be the set of infinite random walk trajectories $(p_1,\dots ,p_d)$ for which 
\begin{equation}\label{eq:748}
    \mathbb P \big( |A\cap [0,t]| \ge t^{1-\alpha /2 } \ | \ \forall i, \   \tilde{X}_i=p_i   \big) \le t^{-d}.
\end{equation}
We have that 
\begin{equation}\label{eq:P_t}
    \mathbb P \big( (\tilde{X}_1,\dots ,\tilde{X}_d)\in \mathcal P _t  \big) \ge 1-t^{-d}.
\end{equation}

\begin{lem}\label{lem:poisson clocks}
Let $(p_1,\dots ,p_d)\in \mathcal P _t$ and recall that $N_i(t)$ is the number of steps taken by the walk $X_i$ up to time $t$. For any $n_1,\dots ,n_d $ we have that 
\begin{equation}
    \mathbb P \left(\forall i,\  N_i(t)=n_i \ \big| \ \forall i, \ \tilde{X}_i=p_i \right)\le t^{-\alpha (d-1)/6-1/2}.
\end{equation}
\end{lem}

\begin{proof}
Let $\tilde{N}_i$ for $i\le d$ be Poisson processes independent of each other and independent of $\tilde{X}_i$ for all $i$. We couple $N_i$ and $\tilde{N}_i$ in the following way. For any $s\notin A$ we let $N_i(s)$ jump whenever $\tilde{N}_i(s)$ jumps and for $s\in A$ we let $N_i$ jump according to Poisson clocks on the vertices that are independent of $\tilde{N}_i$ and $\tilde{X}_i$ for $i\le d$. It is clear that the processes $N_i$ defined in this way have the right law. Moreover, for all $i\neq j$ the Poisson processes $\tilde{N}_i$ and $N_j$ are independent.

Next, note that the process $N_i(s)-\tilde{N}_i(s)$ is a martingale and its predictable quadratic variation is $|A\cap [0,s]|$. Thus, by Freedman's inequality (see \cite[Lemma 2.1]{van1995exponential}) we have that 
\begin{multline}\label{eq:freedman}
    \mathbb P \left( |N_i(t)-\tilde{N}_i(t)|\ge t^{1/2-\alpha /6}   \ \text{ and } \  |A\cap [0,t]|\le t^{1-\alpha /2}  \ \big| \ \forall i, \ \tilde{X}_i=p_i  \right) \\
    \le \exp\left(-\frac{t^{1-\alpha /3}}{2(t^{1/2-\alpha /6} +t^{1-\alpha /2})}\right)\leq \exp (-t^{\alpha /6}/3)\le t^{-2d}.
\end{multline}
Therefore, by the definition of $\mathcal P _t$ we have 
\begin{equation}
    \mathbb P \left( \exists i, \ |N_i(t)-\tilde{N}_i(t)|\ge t^{1/2-\alpha /6 }  \ \big| \ \forall i, \ \tilde{X}_i=p_i  \right) \le 2t^{-d}.
\end{equation}
It follows that
\begin{equation}
\begin{split}
    &\mathbb P \left(\forall i,\  N_i(t)=n_i \ \big| \ \forall i, \ \tilde{X}_i=p_i \right) \\
    &\quad \quad \le \mathbb P \left(N_1(t)=n_1, \  \forall i\ge 2, \ |\tilde{N}_i(t)-n_i|\le t^{1/2-\alpha /6 } \ \big| \ \forall i, \ \tilde{X}_i=p_i \right)  \\
    &\quad \quad\quad \quad\quad \quad\quad \quad\quad \quad\quad \quad\quad \quad +\mathbb P \left( \exists i\ge 2, \ |N_i(t)-\tilde{N}_i(t)|\ge t^{1/2-\alpha /6 }     \ \big| \ \forall i, \ \tilde{X}_i=p_i \right) \\
    &\quad \quad \le  \mathbb P \big( N_1(t)=n_1 \big) \prod _{i=2}^d \mathbb P \left( |\tilde{N}_i(t)-n_i|\le t^{1/2-\alpha /6}    \right)  +2t^{-d} \\
    &\quad \quad \le t^{-1/2}\big(  t^{-\alpha /6 }/2  \big)^{d-1} +2t^{-d} \le t^{-\alpha (d-1)/6-1/2}.
\end{split}
\end{equation}
where in the second to last inequality we used that $\mathbb P (\text{Poisson} (t)=n) \le 1/(2\sqrt{t})$ for all sufficiently large $t$ and all $n\in \mathbb N$. This finishes the proof of the lemma.
\end{proof}

In what follows it is slightly easier to work with finite trajectories rather than infinite trajectories. We let $\mathcal P _t'$ be the set of trajectories $(p_1',\dots ,p_d')$ of length $\lfloor 2t \rfloor $ that can be extended to trajectories $(p_1,\dots ,p_d)\in \mathcal P _t$. We claim that by Lemma~\ref{lem:poisson clocks}, for all $(p_1,\dots ,p_d)\in \mathcal P_t'$ and all $n_1,\dots ,n_d\le 2t$ we have that 
\begin{equation}
    \mathbb P \left(\forall i,\  N_i(t)=n_i \ \big| \ \forall i, \ \tilde{X}_i=p_i \right)\le t^{-\alpha (d-1)/6-1/2},
\end{equation}
where in here $\tilde{X}_i=p_i$ is a shorthand for $\tilde{X}_i(n)=p_i(n)$ for all $n\le 2t$. Indeed, the event $\big\{ \forall i , \ N_i(t)=n_i \big\}$ is independent of the trajectories $\tilde{X}_i$ after $2t$ steps and therefore, when estimating the conditional probability, these trajectories an be extended arbitrarily. 

The stage is now ready for the proof of Proposition~\ref{prop:fragmentation infinite-graph d-th moment}.
\begin{proof}[Proof of Proposition~\ref{prop:fragmentation infinite-graph d-th moment}]
Let $I_t:=\big[ t-t^{1/2+\alpha /18  } , t+t^{1/2+\alpha /18} \big] $ and note that 
\begin{equation*}
  \mathbb P (\exists i, \ N_i(t)\notin I_t)\le d\exp (-ct^{\alpha /9}) \le t^{-d}.  
\end{equation*}
Thus, by \eqref{eq:P_t} and Lemma~\ref{lem:poisson clocks} we have that 
 \begin{equation*}
     \begin{split}
     \mathbb P \big( X_1(t)=\cdots =X_d(t) \big) &\le \mathbb P \big( \exists i , \ N_i(t)\notin I_t \big) +\mathbb P \big(  (\tilde{X}_1,\dots ,\tilde{X}_d)\notin \mathcal P _{t}'  \big) \\
    &\!\!\!\!\!\!\!\!\!\!\!\!\!\!\!\!\!\!\!\!\!\!\!\!\!\!\!\!\!\!+\sum _{n_1,\dots ,n_d\in I_t} \sum _{ \substack{(p_1,\dots ,p_d)\in \mathcal P_{t}'  \\
    p_1(n_1)=\cdots =p_d(n_d) } } \mathbb P \big( \forall i , \ \tilde{X}_i=p_i \big) \cdot  \mathbb P \left(\forall i,\  N_i(t)=n_i \ \big| \ \forall i, \ \tilde{X}_i=p_i \right) \\
    &\leq 2t^{-d}  +t^{-\alpha (d-1)/6-1/2} \sum _{n_1,\dots ,n_d\in I_t} \sum _{ \substack{(p_1,\dots ,p_d)\in \mathcal P_{t}'  \\
    p_1(n_1)=\cdots =p_d(n_d) } } \mathbb P \big( \forall i , \ \tilde{X}_i=p_i \big) 
    \\
    &\le 2t^{-d} +  t^{-\alpha (d-1)/6-1/2}
    \sum _{n_1,\dots ,n_d\in I_t} \mathbb P \big( \tilde{X}_1(n_1)=\cdots =\tilde{X}_d(n_d) \big)
    \\
    &\le 2t^{-d}  +  t^{-\alpha (d-1)/6-1/2} |I_t|^d \big( 2C_0 t^{-1/2}\big) ^{(d-1)}\le t^{-\alpha d/10} ,
     \end{split}
 \end{equation*}
 where in the fourth inequality we used Assumption~\ref{assumption 1/sqrt t} and the fact that $\tilde{X}_i$ are independent. This finishes the proof of the proposition using Claim~\ref{claim:1}.
\end{proof}

\section{IID Initial Opinions}

In this section we work under the assumption that the initial opinions are bounded i.i.d.\ random variables. 

\begin{assumption}[i.i.d.\ initial opinions]\label{assumtion iid}
The initial opinions are i.i.d.\ random variables with expectation $\mu $ and distribution supported on $[0,1]$.
\end{assumption}

We note that the interval $[0,1]$ in Assumption~\ref{assumtion iid} can be replaced with any other bounded interval. Indeed, this follows by re-scaling and shifting all the opinions. Furthermore, even a compact support is not crucial. Most of our arguments work for distributions with finite exponential moment such as normal, exponential, etc.  For simplicity, we did not work in the most general settings.

Under Assumption~\ref{assumtion iid}, we derive the next two theorems regarding convergence to consensus on finite and infinite networks respectively. 

\begin{thm}\label{thm:polylog}
Suppose that Assumptions~\ref{assumption 1/sqrt t} and~\ref{assumtion iid}  
 hold and $|V|=n$ is sufficiently large (depending on $C_0$). Then, there exists a universal constants $C,c>0$ such that for any $\eps \ge n^{-c}$ we have
\begin{equation}
  \mathbb P \big( \tau_\eps \ge  (\eps ^{-1} \log n)^C \big) \le e^{-\log ^2 n}.
\end{equation}
\end{thm}

\begin{thm}\label{thm: convergence infinite network}
If $V$ is infinite and Assumptions~\ref{assumption 1/sqrt t} and~\ref{assumtion iid} hold,
then $\lim_{t\to\infty}f_t(v)=\mu$ almost surely, for every $v\in V$.
\end{thm}

\begin{remark}\label{remark: polylog expectation}
Theorem~\ref{cor: infinite graph} follows from Theorem~\ref{thm: convergence infinite network} using the fact that a simple random walk on an infinite connected graph of bounded degree satisfies Assumption~\ref{assumption 1/sqrt t} (see Remark~\ref{remark:assumption}). 

Next, we explain how Theorem~\ref{cor:polylog} follows from Theorem~\ref{thm:polylog}. Let $C,c$ be the constants from Theorem~\ref{thm:polylog} and let $t_1:=(\eps ^{-1}\log n)^C$. Suppose first that $\eps \ge n^{-c}$ and note that by Theorem~\ref{thm:polylog} we have $\mathbb P (\tau _{\epsilon }\ge t_1)\le e^{-\log ^2 n}$. Next, observe that at time $t_1$ all the opinions are in $[0,1]$ and therefore, on the event $\{\tau _\eps \ge t_1\}$, we can use either Theorem~\ref{cor:yuval} in the undirected case or Corollary~\ref{cor:euler} in the Eulerian directed case to bound the expected additional time required after $t_1$ to reach an $\eps $ consensus by $n^4\log(1/\eps)$, for all sufficiently large $n$. We get
\begin{equation}
\begin{split}
    \mathbb E [\tau _\eps ]&=\mathbb E [ \mathbb E  [\tau _\eps \ |  \ \mathcal G _{t_1}] ] \le t_1 \mathbb P (\tau _{\eps }\le t_1)  +\mathbb E \big[ \mathds 1 \{\tau _{\eps }> t_1 \} \cdot  \mathbb E [\tau _{\eps } \ | \ \mathcal G _{t_1}]  \big] \\
    &\le  t_1\mathbb P (\tau _\eps \le t_1) +\big( t_1+n^4\log (1/\eps ) \big) \mathbb P (  \tau _{\eps }> t_1) \le (\eps ^{-1}\log n)^{\tilde{C}}, 
\end{split}
\end{equation}
for some universal constant $\tilde{C}$ where $\mathcal G _t$ is the sigma algebra generated by the process up to time $t$. Next, suppose that $\eps \le n^{-c}$. In this case Theorem~\ref{cor:polylog} follows immediately from Theorem~\ref{cor:yuval} and Corollary~\ref{cor:euler}. 
\end{remark}

Both Theorem~\ref{thm:polylog} and Theorem~\ref{thm: convergence infinite network} rely on the following crucial lemma. Lemma~\ref{lem: concentration around mu} below estimates the rate in which the distribution of the opinions concentrates around $\mu$.

\begin{lem}\label{lem: concentration around mu}
Under Assumptions~\ref{assumption 1/sqrt t} and \ref{assumtion iid}, there exists $t_0>0$ (depending on $C_0$) and  a universal constant $\beta >0$ such that for every $t\in (t_0,|V|^2/3)$, every $\eps >0$, and every $o\in V$
\[
    \mathbb P \big( |f_t(o)- \mu  |\ge \eps   \big) \le 2\exp \big( -\eps ^2t^{\beta } \big) .
\]
\end{lem}

\begin{proof}[Proof of Lemma~\ref{lem: concentration around mu}]
Without loss of generality, suppose that $\eps <1/2$. Let $o\in V$ and consider the fragmentation process $m_t(v)$ originating from $o$. Recall that 
\begin{equation}
    f_t(o)\overset{d}{=} \sum _{v\in V} m_t(v)f_0(v).
\end{equation}
Note that the fragmentation process is independent of the initial opinions.

Next, let $\alpha >0$ be the constant from Proposition~\ref{prop:fragmentation infinite-graph second moment}. Let $t>1$ sufficiently large, $d:=t^{\alpha /10}$ and $a:=e^{-1}t^{-\alpha /10}$.

By Markov's inequality we have 
\begin{equation}
    \mathbb P \big( \exists v, m_t(v)\ge a \big) \le \sum _{v\in V} \mathbb P \big( m_t(v)\ge a \big) \le a^{-d}\sum _{v\in V}  \mathbb E [m_t(v)^d] \le a^{-d}t^{-\alpha d/10}\le \exp (-t^{\alpha /10}),
\end{equation}
where in the third inequality we used Proposition~\ref{prop:fragmentation infinite-graph d-th moment} and Proposition~\ref{prop:fragmentation infinite-graph second moment}.

Let $\mathcal F $ be the sigma algebra generated by the fragmentation process. On the event $\mathcal A := \big\{ \forall  v, m_t(v)\le a \big\}\in \mathcal F $ we have that 
\begin{equation}
    \sum _{v\in V} m_t(v)^2 \le a\sum _{v\in V} m_t(v)=a
\end{equation}
and therefore, by Azuma's inequality, on this event we have
\begin{equation}
    \mathbb P \big( |f_t(o)- \mu  |\ge \eps  \ \big| \ \mathcal F  \big) \le \exp \bigg( \frac{-\eps ^2 }{2\sum _{v\in V} m_t(v)^2} \bigg)  \le  \exp \big( -\eps ^2 /(2a) \big)\le \exp (-\eps ^2t^{\alpha /10})  .
\end{equation}
Thus,
\begin{equation}
    \mathbb P \big( |f_t(o)- \mu  |\ge \eps   \big) \le \mathbb P (\mathcal A ^c)+\mathbb P \big( |f_t(o)- \mu  |\ge \eps , \ \mathcal A  \big)   \le 2\exp \big( -\eps ^2t^{\alpha /10} \big) 
\end{equation}
as needed
\end{proof}

\begin{proof}[Proof of Theorem~\ref{thm:polylog}]
Let $\beta$ be the universal constant from Lemma~\ref{lem: concentration around mu} and let $\eps \ge n^{-\beta /4}$. By Lemma~\ref{lem: concentration around mu}, and a union bound we obtain 
\begin{equation}
    \mathbb P (\tau _\eps >t) \le \mathbb P \big( \exists o \in V, \ |f_t(o)- \mu  |\ge \eps /2  \big) \le  2n \exp (-\eps ^2t^{\beta }/4),
\end{equation}
for any $t\in(t_0,n^2/3)$, where $t_0$ depends only on $C_0$. Letting  $t_1:=(\eps ^{-1}\log n)^{4/\beta } \le n^2/3$ we have $\mathbb P (\tau _\eps >t_1) \le e^{-\log ^2 n}$.
\end{proof}

\begin{proof}[Proof of Theorem~\ref{thm: convergence infinite network}]
Fix $\eps>0$ and $v\in V$. Suppose for the sake of contradiction that $\mathbb P(\limsup_{t\to\infty} |f_t(v)-\mu|>\eps)=c_1>0$. For $n\in \mathbb N$, let $t_n:=\inf\{t> n:|f_t(v)-\mu|>\eps\}$ and consider the event $A_n$ that $t_n<\infty$ and the clock at $v$ did not ring after time $t_n$ up to time $\lceil t_n\rceil$. On one hand, we have $\mathbb P(A_n)\geq c_1e^{-1}>0
$. On the other hand, 
\[\mathbb P(A_n)\leq \sum_{k=n}^\infty \mathbb P(|f_k(v)-\mu|>\eps)\]
which converges to zero as $n$ goes to infinity, by Lemma~\ref{lem: concentration around mu}. 
\end{proof}

\section{Tightness of the Results}
\label{section:counterexamples}

In this section, we provide examples to demonstrate the tightness of our results. The matrix $P$ in all these examples is the transition matrix of a simple random walk on a graph.

\subsection{Lower bounds for the consensus time with arbitrary initial opinions}
\begin{claim}\label{claim:lower bound spectral}
Suppose that $P$ is the transition matrix of a simple random walk on a graph $G=(V,E)$. There are initial opinions $f_0(v)\in [0,1]$ such that $\mathbb E [\tau _{1/2}]\ge 1/(10\gamma )$, where $\gamma $ is the spectral gap of $P$.
\end{claim}

\begin{proof}
Let $g\colon V\to \mathbb R$ be the right eigenvector corresponding to the second largest eigenvalue of $P$. We define the vector of initial opinions as follows 
\begin{equation}
    f_0(v):=\frac{g(v)-\min _u g(u)}{\max _u g(u)-\min _u g(u)}.
\end{equation}
Note that $\min _v f_0(v)=0$ and $\max _v f_0(v)=1$. Moreover, note that $f_0$ is a right eigenvector of the matrix $e^{t(P-I)}$ with eigenvalue $e^{t(\lambda _2-1)}=e^{-\gamma t}$ for all $t>0$.
Thus, by \eqref{eq:degroot fragmentation} and \eqref{def:frag} we have that 
\begin{equation}
\begin{split}
    \mathbb E [f_t(v)]&=\sum _u \mathbb E [m_t(v,u)]f_0(u)=\sum _u  \mathbb P _v (X(t)=u)f_0(u)\\
    &=\sum _u   \big(e^{t(P-I)} \big) _{v,u} f_0(u)=\big( e^{t(P-I)}f_0 \big) _v =e^{-\gamma t}f_0(v).
\end{split}
\end{equation}
Thus, letting $v_1,v_2\in V$ such that $f(v_2)-f(v_1)=1$ and $t_1:=1/(2\gamma )$ we obtain $\mathbb E [f_{t_1}(v_2)-f_{t_1}(v_1)]\ge e^{-1/2}$. Since $f_{t_1}(v_2)-f_{t_1}(v_1)\le 1$, it follows that 
\begin{equation}
    \mathbb P \big( \tau _{1/2} \ge t_1 \big)  \ge \mathbb P  \big( f_{t_1}(v_2)-f_{t_1}(v_1) > 1/2 \big) \ge 2e^{-1/2}-1.
\end{equation}
This finishes the proof of the claim.
\end{proof}

\begin{claim}\label{claim:exp}
For all $n$ sufficiently large, there exists a network of size $n$ and initial opinions in $[0,1]$ for which $\mathbb E [\tau _{1/2} ]\ge e^{cn}$.
\end{claim}

\begin{proof}
We let $P$ be the transition matrix of a simple random walk on the directed graph $G=(V,E)$ defined as follows. We let $V:=\{-n, -n+1,\dots ,n\}$. For $-n+2\le i <0$ the out edges of $i$ are $(i,i+1)$, $(i,i-1)$ and $(i,i-2)$. For $0<i\le n-2$ the out edges of $i$ are $(i,i-1)$, $(i,i+1)$ and $(i,i+2)$. We also add the following edges to make the graph strongly connected: $(0,-1)$, $(0,1)$,$(-n,-n+1)$, $(-n+1,-n)$, $(-n+1,-n+2)$, $(n,n-1)$, $(n-1,n)$ and $(n-1,n-2)$. The idea in here is that the random walk on $G$ will have a negative drift in the left half of the interval and a positive drift in the right half of the interval. We define the initial opinions as follows 
\begin{equation}
f_0(i):=
    \begin{cases}
     0 \quad  i< 0  \\
    1 \quad i\ge 0 
    \end{cases}.
\end{equation}
Using the same arguments as in the proof of Claim~\ref{claim:lower bound spectral} we have that $\mathbb E [f_{t}(-n)]\le \mathbb P (X(t)\ge 0)$, where $X(t)$ is the simple random walk on $G$ starting from $-n$. Thus, using that the drift of $X(t)$ is negative when $X(t)<0$ we obtain that $\mathbb E [f_{t_1}(-n)] \le 1/16$, where $t_1:=e^{c_0n}$ and where $c_0>0$ is sufficiently small. By Markov's inequality we have $\mathbb P \big( f_{t_1}(-n) \ge 1/4 \big) \le 1/4$. Finally, by symmetry we have $\mathbb P \big( f_{t_1}(n) \le 3/4 \big) \le 1/4$ and therefore 
\begin{equation}
    \mathbb P (\tau _{1/2} \le t_1 ) \le \mathbb P \big( |f_{t_1}(-n)-f_{t_1}(n)|\le 1/2  \big) \le 1/2.
\end{equation}
This finishes the roof of the claim.
\end{proof}

\subsection{A lower bound for the consensus time with i.i.d.\ initial opinions}
Theorem~\ref{thm:polylog} states that if the initial opinions are bounded i.i.d.\ random variables then the $\eps$-consensus time grows at most polylogarithmically in $n$ and polynomially in $1/\eps $. In the following claim we consider the special case in which $P$ is the transition matrix of a simple random walk on the cycle graph and the initial opinions are normal. In this case, we give a lower bound that is polylogarithmic in $n$ and polynomial in $1/\eps $. This shows that Theorem~\ref{thm:polylog} (or Theorem~\ref{cor:polylog}) cannot be improved in general. 

\begin{claim}\label{claim:lower bound}
Let $n\ge 1$ sufficiently large and let $P$ be the transition matrix of a simple random walk on the cycle graph of length $n$. Suppose that initially $f_0(v)$ are i.i.d.\ standard normal random variables. Then, there exists a universal constant $c>0$ such that, for any $n^{-1/9} \le \eps <1/2$, we have 
\begin{equation}
    \mathbb E [\tau _{\eps }] \ge c\eps ^{-4} \log ^2 n.
\end{equation}
 \end{claim}

For the proof of the claim we will need the following lemma.

\begin{lem}\label{lem:prob12}
For any vertex $v$ in the cycle and for all $t\le \sqrt{n}$ and $\eps \ge t^{-1/4} $ we have that 
\begin{equation}
    \mathbb P (f_t(v)\ge \eps ) \ge e^{-C\eps ^2 \sqrt{t}}
\end{equation}
\end{lem}

\begin{proof}
Let $m_t$ be the fragmentation process originating at $v$. Recall that $\mathbb E [m_t(u)]$ is the probability that a simple continuous time random walk starting from $v$ will reach $u$ at time $t$. Thus, there exists $C_2>0$ such that 
\begin{equation}
    \sum _{d(u,v)\le C_2\sqrt{t}} \mathbb E [m_t(u)]\ge \frac{3}{4},
\end{equation}
where the sum is over vertices $u$ in the cycle that are at graph distance at most $C_2\sqrt{t}$ from $v$. Since the total mass is $1$ it follows that 
\begin{equation}\label{eq:sum of mass}
    \mathbb P \bigg( \sum _{d(u,v)\le C_2\sqrt{t}} m_t(u) \ge \frac{1}{2} \bigg)\ge \frac{1}{2}.
\end{equation}
    On the event in the left hand side of \eqref{eq:sum of mass}, we have by Cauchy Schwarz inequality   
    \begin{equation}
        \frac{1}{4} \le \bigg( \sum _{d(u,v)\le C_2\sqrt{t}} m_t(u) \bigg)^2 \le 2C_2\sqrt{t} \sum _{d(u,v)\le C_2\sqrt{t}} m_t(u)^2 \le
        2C_2\sqrt{t} \sum _{u\in V}m_t(u)^2.
    \end{equation}
We obtain that there exists $C_3>0$ such that the event 
\begin{equation}
    \mathcal A := \bigg\{ \sum _{u\in V }m_t(u)^2 \ge \frac{1}{C_3\sqrt{t}} \bigg\} 
\end{equation}
holds with probability at least $1/2$.

Next, note that conditioning on the fragmentation process we have that
\begin{equation}
    f_t(v)=\sum _{u\in V} m_t(u)f_0(u) \sim N \Big( 0\ ,\  \sum m_t(u)^2 \Big).
\end{equation}
Thus, letting $\mathcal F$ be the sigma algebra generated by the fragmentation process and using the fact that the tail $\mathbb P (X>x)$ for $X\sim N(0,\sigma ^2)$ increases with $\sigma $ we obtain that on $\mathcal A $
\begin{equation}
    \mathbb P \big( f_t(v)\ge \eps \ | \ \mathcal F \big) \ge e^{-C\eps ^2 \sqrt{t}}
\end{equation}
This finishes the proof of the lemma using that $\mathbb P (\mathcal A )\ge 1/2$.
\end{proof}

We can now prove Claim~\ref{claim:lower bound}.

\begin{proof}[Proof of Claim~\ref{claim:lower bound}]
Recall the definition of the fragmentation process $m_s^t(v,u)$ given in the discussion before equation~\eqref{eq:m_s^t}. This process uses the clocks on the vertices in the DeGroot dynamics going backward in time from $t$ to $0$. Moreover, recall that in this coupling we have the identity $f_t(v)=\sum _u m_t^t (v,u)f_0(u)$.

Let $t_1:=\delta \eps ^{-4} \log ^2 n$ where $\delta >0$ is sufficiently small (independently of $\eps $ and $n$) and will be determined later. For a vertex $v$ in the cycle consider the event 
\begin{equation}
    \mathcal B _v :=\Big\{ \text{For all } t\le t_1 \text{ and any }u\in V \text{ with } d(u,v)\ge 2t_1 \text{ we have } m_{t}^{t_1}(v,u) =0 \Big\}.
\end{equation}
Note that the event $\mathcal B _v$ depends only on the Poisson clocks of the vertices at distance at most $2t_1$ from $v$. Moreover, the compliment of the event $\mathcal B_v$ occurs only if there is a decreasing sequence of rings on $\lfloor 2t_1 \rfloor $ consecutive vertices along the cycle starting from $v$. It follows that $\mathbb P (\mathcal B _v )\ge 1-Ce^{-ct_1}$. 

Define $\mathcal C _v:=\mathcal B _v\cap \{f_{t_1}(v)\ge \eps \}$ and note that by Lemma~\ref{lem:prob12} we have 
\begin{equation}\label{eq:prob13}
    \mathbb P (\mathcal C _v) \ge e^{-C\eps ^2 \sqrt{t_1}} \ge e^{-C\delta \log n}\ge n^{-1/3},
\end{equation}
where the last inequality holds for as long as $\delta $ is sufficiently small. An important observation is that the event $\mathcal C_v$ depends only on the initial opinions and the Poisson clocks of the vertices that are at distance at most $2t_1$ from $v$. 

Let $k:=\lfloor \sqrt{n} \rfloor -1$ and let $v_1,\dots v_k$  be a sequence of (roughly equally spaced) vertices along the cycle such that $d(v_i,v_j)\ge \sqrt{n}$ for any $i$ and $j$. Using that $\eps \ge n^{-1/9}$ we get that $4t_1 \le \sqrt{n}$ and therefore $\mathcal C _{v_i}$ are mutually independent. Thus by \eqref{eq:prob13} we have that 
\begin{equation}
  \mathbb P \big( \exists  v, \ f_{t_1}(v)\ge \eps \big) \ge \mathbb P \big( \exists i \le k , \ \mathcal C _{v_i} \text{ holds} \big) \ge 1-\big( 1-n^{-1/3} \big)^k \ge   3/4.  
\end{equation}
By symmetry we have $\mathbb P \big( \exists  v, \ f_{t_1}(v) \le -\eps  \big)\ge 3/4$ and therefore $\mathbb P (\tau _{\eps } \ge t_1 ) \ge 1/2 $. This finishes the proof of the claim.
\end{proof}

\begin{claim}\label{claim:n^2}
There exists a graph of size $n$ (of large maximal degree) such that starting from i.i.d.\ uniform in  $\{-1,1\}$ initial opinions,
\[
\mathbb E[\tau_{1/4}]=\Omega(n^2).
\]
\end{claim}

\begin{proof}[Sketch of proof]
The graph is defined as follows. Let $S_1$ be a star graph with $n$ leafs and center $v_1$ and $S_2$ an identical copy of $S_1$ with center $v_2$. The centers $v_1$ and $v_2$ are connected by a path of length $n$.

Let $t=\delta n^2$, where $\delta>0$ is a constant that will be determined later. We claim that  
\begin{equation}
  \mathbb E \big[ f_{t}(v_1)-f_{t}(v_2) \ | \ f_0(v_1)=1, \ f_0(v_2)=-1\big]\ge \tfrac 1 2.
\end{equation}
Indeed, by \eqref{def:frag} and 
\begin{equation}
  \mathbb E \big[ f_{t}(v_1) \ | \ f_0(v_1)=1,\ f_0(v_2)=-1\big]=\mathbb P_{v_1}(X(t)=v_1)-\mathbb P_{v_1}(X(t)=v_2)\geq \pi_{v_1}- \kappa(\delta)\ge \tfrac 1 4,
\end{equation}
where $\kappa(\delta)>0$ satisfies $\kappa (\delta )\to 0$ as $\delta \to 0$. The last inequality holds as long as $\delta $ is sufficiently small and using that $\pi_{v_1}=\frac {n+1}{3n}>\frac 1 3$.
Similarly,
\[
\mathbb E \big[ f_{t}(v_2) \ | \ f_0(v_1)=1,\ f_0(v_2) \big] \leq  - \tfrac 1 4.
\]

From the above claim, it follows that $\mathbb P(|f_t(v_1)-f_t(v_2)|>1/4)>c$, for some constant $c>0$.
\end{proof}

\subsection{Examples of non-convergence of opinions}
In this section we show that Assumptions~\ref{assumption 1/sqrt t} and~\ref{assumtion iid} in Theorem~\ref{thm: convergence infinite network} are crucial. We remark that our examples apply for the synchronous DeGroot model as well. We construct two examples: the first violates only Assumption~\ref{assumption 1/sqrt t} and the second violates only Assumption~\ref{assumtion iid}. 

In both of these examples the initial opinions $f_0$ will be bounded (in $[-1,1]$) and therefore $f_t$ as well. It follows that, for every $v\in V$, the convergence $\lim_{t\to\infty}f_t(v)$ in probability is equivalent to convergence in $L_1$ and $L_2$. Hence, it is sufficient to find examples for which there exists $v\in V$ such that $\mathbb E[f_t(v)]$ does not converge. By the definitions of the DeGroot dynamics \eqref{eq:degroot fragmentation} and the fragmentation process \eqref{def:frag}, $\mathbb E[f_t(v)|\text{initial opinions}]= \sum_{u\in V}P_{v,u}^t f_0(u)$, where $P^t$ is the $t$-fold transition matrix of a continuous time Markov chain on $V$.

\begin{claim}\label{claim:unbounded}
There exists a network with bounded i.i.d.\ initial opinions and a vertex $v$ such that the opinion of $v$ diverges in probability. 
\end{claim}
\begin{proof}
The matrix $P$ is the transition matrix of a simple random walk on a graph $G=(V,E)$ constructed as follows. The graph contains an infinite one-sided line. Each vertex $v_i$ on the line ($i\in \mathbb N$) is further connected to a set $L_i$ of $N_i=|L_i|$ leaves. 

The numbers $N_1,N_2,\ldots$ are defined recursively together with times $t_1<t_2<\cdots$ such that the following property holds. Let $X(t)$ be the continuous time random walk on $G$ originating at $v_1$. Define hitting times $T_i:=\inf\{t:X(t)=v_i\}$ and events \[\mathcal E_{i}:=\{T_i \leq t_i, \forall t\in [T_i,t_i],\ X(t)\in L_i\cup\{v_i\}\}.\] 
The following property should hold: 
\begin{equation}\label{eq:stars}
    \mathbb P(\mathcal E_i|\mathcal E_1,\ldots,\mathcal E_{i-1})\geq \exp\left( \frac{-\delta}{2^{i}}\right),
\end{equation}
where $\delta>0$ is a constant that will be fixed later.

Define $N_1:=0$ and $t_1:=0$. Suppose $N_1,\ldots N_k$ and $t_1,\ldots,t_k$ are already defined. Note that $T_{k+1}$ is already well defined. Let $t_{k+1}$ be so large that $\mathbb P(T_{k+1} \leq t_{k+1}|\mathcal E_1,\ldots,\mathcal E_k)\geq \exp(-\frac{\delta}{2^{k+2}})$, and in addition, $\mathbb P(N(t_{k+1})=k\mod 2))\geq \frac 1 3$, where $N(t)\sim \mathrm{Poisson}(t)$. Define $N_{k+1}$ so large that a random walk that starts in $v_{k+1}$ does not leave $L_{k+1}\cup\{v_{k+1}\}$ until time $t_{k+1}$ with probability at least $\exp(-\frac{\delta}{2^{k+2}})$. It follows that \eqref{eq:stars} is  satisfied, and in addition
\[
\mathbb P(N(t_i)\neq i\mod 2)\geq \tfrac 1 3.
\]

Now, let the distribution of the initial opinions be uniform in $\{-1,+1\}$ and set $\delta = 0.01$. Let $\mathcal E = \cap_{i}\mathcal E_i$. Then $\mathbb P(\mathcal E)\geq e^{-\delta}$, and 
\[
\mathbb E[m_{t_i}(v_i)]= P(X(t_i)=v_i)\geq \mathbb P(N(t_i)\neq i\mod 2, \mathcal E)\geq e^{-\delta}- 2/3 \geq 1/4 ,
\]
where $m_t$ is the fragmentation process originating at $v_1$.

Let $\mathcal F$ be the sigma-algebra generated by the clock rings. We get 
\begin{equation}\label{eq:cov}
\begin{split}
\mathrm{Cov}(&f_{t_i}(v_1),f_0(v_i))=\mathbb E \big[ \mathbb E \big[ f_{t_i}(v_1)f_0(v_i) \ | \ \mathcal F\big] \big]\\
&=
\mathbb E\bigg[ \mathbb E\bigg[ \sum_{u}m_{t_i}(u)f_0(u)f_0(v_i) \ \Big| \ \mathcal F\bigg] \bigg] = \mathbb E\bigg[ \sum_u m_{t_i}(u) \mathbb E\bigg[ f_0(u)f_0(v_i) \ \Big| \ \mathcal F\bigg] \bigg]\\
&=\mathbb E\bigg[ \sum_u m_{t_i}(u) \mathbb E\big[ f_0(u)f_0(v_i)\big] \bigg] =\mathbb E\left[m_{t_i}(v_i)\right]\geq \frac 1 4.
\end{split}
\end{equation}
If $f_t(v_1)$ converged to some random variable $f_{\infty }(v_1)$ in probability, then it would also converge in $L_2$. Therefore,  $\mathrm{Cov}(f_{\infty }(v_1),f_0(v_i))\geq 1/5$, for every $i$ large enough. Since $f_0(v_i)$ are independent, we get $\mathrm{Var}(f_{\infty}(v_1))\geq\sum_{i}\mathrm{Cov}(f_{\infty}(v_1),f_0(v_i))=\infty$. Contradiction, since $f_{\infty}(v_1)\in[-1,1]$. 
\end{proof}

The second example shows that the assumption of i.i.d.\ initial opinions is crucial.

\begin{claim}\label{claim:initial}
There are bounded initial opinions $f_0(v)$ on the line graph $\mathbb Z $ such that $\mathbb E f_t(0)$ does not converge.
\end{claim}

\begin{proof}
For any $k\ge 1$ define the interval $I_k:=\big[ -4^{4^k},4^{4^k} \big] $. For a vertex $v\in I_k \setminus I_{k-1}$, if $k$ is odd we let $f_0(v):=0$ and if $k$ is even we let $f_0(v):=1$. We claim that with these initial opinions $\mathbb E [f_t(0)]$ does not converge. To this end, consider the sequence of times $t_k:=4^{4^k}$. Suppose that $k$ is odd, and let $X_t$ be the simple random walk on $\mathbb Z$. We have that 
\begin{equation}
\begin{split}
    \mathbb E [f_{t_k}(0)]\le \mathbb P (X_{t_k} \in I_{k-1} ) + \mathbb P (X_{t_k} \notin I_{k} ) &\le Ct_k^{-1/2} \cdot |I_{k-1}| +\mathbb P \big( |X_{t_k} |\ge t_k  \big)\\
    &\le C4^{-2\cdot 4^{k-1}} 4^{4^{k-1}} +Ce^{-ct_k} \le C4^{-4^{k-1}},
\end{split}
\end{equation}
where in the second inequality we used that $\mathbb P (X_t=v) \le Ct^{-1/2}$ for any $v\in \mathbb Z$ and in third inequality we used Chernoff bound.
Similarly, when $k$ is even we have that 
\begin{equation}
    \mathbb E [f_{t_k}(0)] \ge 1-C4^{-4^{k-1}}.
\end{equation}
This shows that 
\begin{equation}
    0=\liminf _{t\to \infty } \mathbb E [f_t(0)] <\limsup _{t\to \infty } \mathbb E [f_t(0)] =1 
\end{equation}
and in particular $\mathbb E [f_{t_k}(0)]$ does not converge.
\end{proof}
\section{Acknowledgements}
We would like to thank two referees who contributed many useful comments that helped improve the quality of the paper significantly. RP is supported in part by the Israel Science Foundation grant \# 2566/20.
\bibliographystyle{plain}
\bibliography{bibliography}

\end{document}